\documentclass[a4paper,oneside,11pt]{article}

\usepackage[a4paper]{geometry}
\usepackage{aeguill}                    % ... or a4paper or a5paper or ... 
\usepackage{graphicx}
\usepackage{amsmath,amsfonts,amssymb,amsthm}
\usepackage{pstricks}
\usepackage{epstopdf}
\usepackage[utf8]{inputenc} 
\usepackage[T1]{fontenc} 
 
\usepackage[english]{babel} 

\title{The boundary of the outer space of a free product}
\author{Camille Horbez}
	
	\date{}
\usepackage[all]{xy}

\usepackage{bbm}
\begin{document}
\maketitle
%\tableofcontents
\newtheorem{de}{Definition} [section]
\newtheorem{theo}[de]{Theorem} 
\newtheorem{prop}[de]{Proposition}
\newtheorem{lemma}[de]{Lemma}
\newtheorem{cor}[de]{Corollary}
\newtheorem{propd}[de]{Proposition-Definition}

\theoremstyle{remark}
\newtheorem{rk}[de]{Remark}
\newtheorem{ex}[de]{Example}
\newtheorem{question}[de]{Question}

\normalsize

\newcommand{\coucou}[1]{\footnote{#1}\marginpar{$\leftarrow$}}

\addtolength\topmargin{-.5in}
\addtolength\textheight{1.in}
\addtolength\oddsidemargin{-.045\textwidth}
\addtolength\textwidth{.09\textwidth}

\begin{abstract}
Let $G$ be a countable group that splits as a free product of groups of the form $G=G_1\ast\dots\ast G_k\ast F_N$, where $F_N$ is a finitely generated free group. We identify the closure of the outer space $\mathbb{P}\mathcal{O}(G,\{G_1,\dots,G_k\})$ for the axes topology with the space of projective minimal, \emph{very small} $(G,\{G_1,\dots,G_k\})$-trees, i.e. trees whose arc stabilizers are either trivial, or cyclic, closed under taking roots, and not conjugate into any of the $G_i$'s, and whose tripod stabilizers are trivial. Its topological dimension is equal to $3N+2k-4$, and the boundary has dimension $3N+2k-5$. We also prove that any very small $(G,\{G_1,\dots,G_k\})$-tree has at most $2N+2k-2$ orbits of branch points.
\end{abstract}

\section*{Introduction}

Let $G$ be a countable group that splits as a free product $$G=G_1\ast\dots\ast G_k\ast F_N,$$ where $F_N$ denotes a finitely generated free group of rank $N$. We assume that $N+k\ge 2$. A natural group of automorphisms associated to such a splitting is the group $\text{Out}(G,\{G_1,\dots,G_k\})$ consisting of those outer automorphisms of $G$ that preserve the conjugacy classes of each of the groups $G_i$. 

The study of the group $\text{Out}(F_N)$ of outer automorphisms of a finitely generated free group has greatly benefited from the study of its action on Culler and Vogtmann's outer space \cite{CV86}, as well as some hyperbolic complexes. The present paper is a starting point of a work in which we extend results about the geometry of these $\text{Out}(F_N)$-spaces to analogues for free products, with a view to establishing a Tits alternative for the group of outer automorphisms of a free product \cite{Hor14-8}. The second main step towards this will be to define hyperbolic complexes equipped with $\text{Out}(G,\{G_1,\dots,G_k\})$-actions, and give a description of the Gromov boundary of the graph of cyclic splittings of $G$ relative to the $G_i$'s \cite{Hor14-6}. 

The group $\text{Out}(G,\{G_1,\dots,G_k\})$ acts on a space $\mathbb{P}\mathcal{O}(G,\{G_1,\dots,G_k\})$ called \emph{outer space}. This was introduced by Guirardel and Levitt in \cite{GL07}, who generalized Culler and Vogtmann's construction \cite{CV86} of an outer space $CV_N$ associated to a finitely generated free group of rank $N$, with a view to proving finiteness properties of the group $\text{Out}(G,\{G_1,\dots,G_k\})$. The outer space $\mathbb{P}\mathcal{O}(G,\{G_1,\dots,G_k\})$ is defined as the space of all $G$-equivariant homothety classes of minimal \emph{Grushko $(G,\{G_1,\dots,G_k\})$-trees}, i.e. metric simplicial $G$-trees in which nontrivial point stabilizers coincide with the conjugates of the $G_i$'s, and edge stabilizers are trivial.

Outer space can be embedded into the projective space $\mathbb{PR}^G$ by mapping any tree in $\mathbb{P}\mathcal{O}(G,\{G_1,\dots,G_k\})$ to the collection of all translation lengths of elements in $G$. The goal of the present paper is to describe the closure of the image of this embedding. 

The closure of Culler and Vogtmann's classical outer space has been identified by Bestvina and Feighn \cite{BF94} and Cohen and Lustig \cite{CL95}, with the space of projective length functions of minimal, \emph{very small} actions of $F_N$ on $\mathbb{R}$-trees. An $F_N$-tree is \emph{very small} if arc stabilizers are cyclic (possibly trivial) and closed under taking roots, and tripod stabilizers are trivial. 

More precisely, Cohen and Lustig have first proved \cite[Theorem I]{CL95} that $\overline{CV_N}$ is contained in the space of projective length functions of very small $F_N$-actions on $\mathbb{R}$-trees. In addition, they have shown that every simplicial, very small $F_N$-tree is a limit of free and simplicial actions \cite[Theorem II]{CL95}. Bestvina and Feighn have shown that this remains true of every very small (possibly nonsimplicial) $F_N$-action on an $\mathbb{R}$-tree. However, their proof does not seem to handle the case of geometric actions that are dual to foliated band complexes, one of whose minimal components is a measured foliation on a nonorientable surface, and in which some arc stabilizer is nontrivial. Indeed, in this case, it is not clear how to approximate the foliation by rational ones without creating any one-sided compact leaf, and one-sided compact leaves are an obstruction for the dual action to be very small (arc stabilizers are not closed under taking roots). Building on Cohen and Lustig's and Bestvina and Feighn's arguments, and using approximation techniques due to Levitt and Paulin \cite{LP97} and Guirardel \cite{Gui98}, we reprove the fact that $\overline{CV_N}$ identifies with the space of minimal, very small projective $F_N$-trees. Our proof tackles both cases of simplicial and nonsimplicial trees at the same time (it gives a new interpretation of Cohen and Lustig's argument in the simplicial case). We work in the more general framework of free products of groups. A \emph{$(G,\{G_1,\dots,G_k\})$-tree} is an $\mathbb{R}$-tree, equipped with a $G$-action, in which all $G_i$'s fix a point. A $(G,\{G_1,\dots,G_k\})$-tree will be termed \emph{very small} if arc stabilizers are either trivial, or cyclic, closed under taking roots, and not conjugate into any of the $G_i$'s, and tripod stabilizers are trivial. We prove the following theorem.

\theoremstyle{plain}
\newtheorem*{theo:1}{Theorem 1}

\begin{theo:1}\label{cv-closure-intro}
The closure $\overline{\mathbb{P}\mathcal{O}(G,\{G_1,\dots,G_k\})}$ in $\mathbb{PR}^G$ is compact, and it is the space of projective length functions of minimal, very small $(G,\{G_1,\dots,G_k\})$-trees.
\end{theo:1}

When $T$ is a $(G,\{G_1,\dots,G_k\})$-tree with trivial arc stabilizers, we can be a bit more precise about the approximations we get, and show that $T$ is an unprojectivized limit of Grushko $(G,\{G_1,\dots,G_k\})$-trees $T_n$, that come with $G$-equivariant $1$-Lipschitz maps from $T_n$ to $T$, see Theorem \ref{Lip-approx}.

We then compute the topological dimension of the closure and the boundary of the outer space of a free product of groups. In the case of free groups, Bestvina and Feighn have shown in \cite{BF94} that $\overline{CV_N}$ has dimension $3N-4$, their result was extended by Gaboriau and Levitt who proved in addition that $\partial CV_N$ has dimension $3N-5$ in \cite{GL95}. Following Gaboriau and Levitt's arguments, we show the following.

\theoremstyle{plain}
\newtheorem*{theo:2}{Theorem 2}

\begin{theo:2}
The space $\overline{\mathbb{P}\mathcal{O}(G,\{G_1,\dots,G_k\})}$ has topological dimension $3N+2k-4$, and $\partial \mathbb{P}\mathcal{O}(G,\{G_1,\dots,G_k\})$ has dimension $3N+2k-5$.
\end{theo:2}

Along the proof, we provide a bound on the number of orbits of branch points and centers of inversion in a very small $(G,\{G_1,\dots,G_k\})$-tree, and on the possible Kurosh ranks of point stabilizers. 

We also introduce the slightly larger class of \emph{tame} $(G,\{G_1,\dots,G_k\})$-trees, defined as those trees whose arc stabilizers are either trivial, or cyclic and non-conjugate into any $G_i$, and with a finite number of orbits of directions at branch points and inversion points. We study some properties of this class, and provide some conditions under which a limit of tame $(G,\mathcal{F})$-trees is tame. This class will turn out to be the right class of trees for carrying out our arguments to describe the Gromov boundary of the graph of cyclic splittings of $G$ relative to the $G_i$'s in \cite{Hor14-6}.  
\\
\\
\indent The paper is organized as follows. In Section \ref{sec-background}, we review basic facts about free products of groups, and the associated outer spaces. We prove in Section \ref{sec-compactness} that the space $VSL(G,\{G_1,\dots,G_k\})$ of projective, minimal, very small $(G,\{G_1,\dots,G_k\})$-trees is compact. In Section \ref{sec-geom}, we introduce a notion of \emph{geometric} $(G,\{G_1,\dots,G_k\})$-trees, and explain in particular how to approximate every very small $(G,\{G_1,\dots,G_k\})$-tree by a sequence of geometric $(G,\{G_1,\dots,G_k\})$-trees. We then compute the topological dimension of $VSL(G,\{G_1,\dots,G_k\})$ in Section \ref{sec-dimension}, and we then identify it with the closure of outer space in Section \ref{sec-closure}. We finally introduce the class of tame $(G,\{G_1,\dots,G_k\})$-trees, and discuss some of its properties, in Section \ref{sec-good}.

\section*{Acknowledgments}

I warmly thank my advisor Vincent Guirardel for his help, his rigour and his patience in reading through first drafts of this work. I am also highly grateful to the referee for their very careful reading of the paper and their numerous and valuable remarks and suggestions. These suggestions led in particular to significant clarifications in the approach to the notion of geometric trees. Finally, I acknowledge support from ANR-11-BS01-013 and from the Lebesgue Center of Mathematics.

\setcounter{tocdepth}{1}
\tableofcontents

\section{Background}\label{sec-background}

\subsection{Free products of groups and free factors}\label{sec-free-products}

Let $G$ be a countable group which splits as a free product of groups of the form $$G=G_1\ast\dots\ast G_k\ast F,$$ where $F$ is a finitely generated free group. We let $\mathcal{F}:=\{[G_1],\dots,[G_k]\}$ be the finite collection of the $G$-conjugacy classes of the $G_i$'s, which we call a \emph{free factor system} of $G$. The rank of the free group $F$ arising in such a splitting only depends on $\mathcal{F}$. We call it the \emph{free rank} of $(G,\mathcal{F})$ and denote it by $\text{rk}_f(G,\mathcal{F})$. The \emph{Kurosh rank} of $(G,\mathcal{F})$ is defined as $\text{rk}_K(G,\mathcal{F}):=\text{rk}_f(G,\mathcal{F})+|\mathcal{F}|$. Subgroups of $G$ that are conjugate into one of the subgroups in $\mathcal{F}$ will be called \emph{peripheral}.

A \emph{splitting} of $G$ is a simplicial tree $S$, equipped with a minimal and simplicial action of $G$ (here minimality of the action means that $S$ contains no proper nonempty $G$-invariant subtree). We say that a subgroup $H\subseteq G$ is \emph{elliptic} in $S$ if there exists a point in $S$ that is fixed by all elements of $H$. A \emph{$(G,\mathcal{F})$-free splitting} is a splitting of $G$ in which all subgroups in $\mathcal{F}$ are elliptic, and all edge stabilizers are trivial. A \emph{$(G,\mathcal{F})$-free factor} is a subgroup of $G$ which is a vertex stabilizer in some $(G,\mathcal{F})$-free splitting.

Subgroups of free products have been studied by Kurosh in \cite{Kur34}. Let $H$ be a subgroup of $G$. Let $T$ be the Bass--Serre tree of the graph of groups decomposition of $G$ represented in Figure \ref{fig-Grushko-rose}, which we call a \emph{standard} $(G,\mathcal{F})$-free splitting. By considering the $H$-minimal subtree in $T$, we get the existence of a (possibly infinite) set $J$, together with an integer $i_j\in\{1,\dots,k\}$, a nontrivial subgroup $H_{j}\subseteq G_{i_j}$ and an element $g_{j}\in G$ for each $j\in J$, and a (not necessarily finitely generated) free subgroup $F'\subseteq G$, so that $$H=\ast_{j\in J}~ g_{j}H_{j}g_{j}^{-1}\ast F'.$$ This decomposition is called a \emph{Kurosh decomposition} of $H$. The \emph{Kurosh rank} of $H$ (which can be infinite in general) is defined as $\text{rk}_K(H):=\text{rk}(F')+|J|$, it does not depend on a Kurosh decomposition of $H$. We let $\mathcal{F}_{|H}$ be the set of all $H$-conjugacy classes of the subgroups $g_{j}H_{j}g_{j}^{-1}$, for $j\in J$, which does not depend on a Kurosh decomposition of $H$ either. 

\begin{figure}
\begin{center}
\input{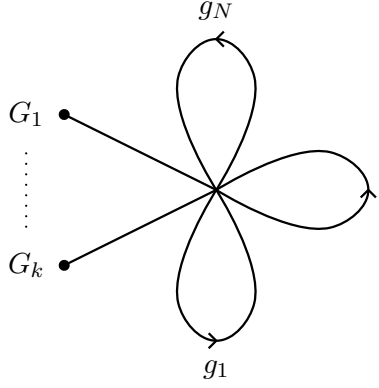}
\caption{A standard $(G,\mathcal{F})$-free splitting.}
\label{fig-Grushko-rose}
\end{center}
\end{figure}

When $H$ is a $(G,\mathcal{F})$-free factor, we have $H_{j}=G_{i_j}$ for all $j\in J$, and all integers $i_j$ are distinct (in particular $J$ is finite). In this case, the free group $F'$ is finitely generated. Hence the Kurosh rank of $H$ is finite.

\subsection{Outer space and its closure}\label{sec-outer-space}

An $\mathbb{R}$-tree is a metric space $(T,d_T)$ in which any two points $x,y\in T$ are joined by a unique embedded topological arc, which is isometric to a segment of length $d_T(x,y)$.

Let $G$ be a countable group, and let $\mathcal{F}$ be a free factor system of $G$. In the present paper, a \emph{$(G,\mathcal{F})$-tree} is an $\mathbb{R}$-tree $T$ equipped with an isometric action of $G$, in which every peripheral subgroup fixes a unique point. A \emph{Grushko $(G,\mathcal{F})$-tree} is a minimal, simplicial metric $(G,\mathcal{F})$-tree with trivial edge stabilizers, whose collection of point stabilizers coincides with the conjugates of the subgroups in $\mathcal{F}$. Two $(G,\mathcal{F})$-trees are \emph{equivalent} if there exists a $G$-equivariant isometry between them.

The \emph{unprojectivized outer space} $\mathcal{O}(G,\mathcal{F})$ is defined to be the space of all equivalence classes of Grushko $(G,\mathcal{F})$-trees. \emph{Outer space} $\mathbb{P}\mathcal{O}(G,\mathcal{F})$ is defined as the space of homothety classes of trees in $\mathcal{O}(G,\mathcal{F})$. We note that in the case where $\mathcal{F}=\{G\}$, outer space is reduced to a single point, corresponding to the trivial action of $G$ on a point.

For all $(G,\mathcal{F})$-trees $T$ and all $g\in G$, the \emph{translation length} of $g$ in $T$ is defined to be $$||g||_T:=\inf_{x\in T}d_T(x,gx).$$

\begin{theo} (Culler--Morgan \cite{CM87}) \label{Culler-Morgan}
The map 

\begin{displaymath}
\begin{array}{cccc}
i:&\mathcal{O}(G,\mathcal{F})&\to &\mathbb{R}^G\\
&T&\mapsto &(||g||_T)_{g\in G}
\end{array}
\end{displaymath}

\noindent is injective.
\end{theo}

We equip $\mathcal{O}(G,\mathcal{F})$ with the topology induced by this embedding, which is called the \emph{axes topology}, and we denote by $\overline{\mathcal{O}(G,\mathcal{F})}$ the closure of the image of this embedding. Culler and Morgan have shown in \cite[Theorem 4.5]{CM87} that if $G$ is finitely generated, then the subspace of $\mathbb{PR}^G$ made of projective classes of translation length functions of minimal $G$-trees is compact. This can fail to be true in general if $G$ is not finitely generated. For example, if $G=A\ast B$ with $A$ not finitely generated, and if $(A_i)_{i\in\mathbb{N}}$ is an increasing sequence of finitely generated subgroups of $A$ whose union equals $A$, then the projective translation length functions of the Bass--Serre trees of the splittings $G=A\ast_{A_i}(B\ast A_i)$ do not admit any converging subsequence in $\mathbb{PR}^G$. However, the key point in Culler--Morgan's argument \cite[Proposition 4.1]{CM87} still holds in our context, in the form the following proposition.

\begin{prop}\label{replace-CM}
Let $R$ be a standard Grushko $(G,\mathcal{F})$-tree, in which all edges are assigned length $1$. Then there exists a finite set $Y\subseteq G$ such that for all $g\in G$ and all trees $T\in\overline{\mathcal{O}(G,\mathcal{F})}$, we have $||g||_T\le M||g||_R$, where $M:=\max_{h\in Y}||h||_T$.
\end{prop}

\begin{proof}
The existence of a finite set $Y\subseteq G$ that works for all trees $T\in\mathcal{O}(G,\mathcal{F})$ is established in \cite[Theorem 4.7]{Hor14-6}. Proposition \ref{replace-CM} then follows by taking limits.
\end{proof}

Arguing as in Culler--Morgan's proof of \cite[Theorem 4.2]{CM87}, we then obtain the following result.

\begin{prop}
The closure $\overline{\mathcal{O}(G,\mathcal{F})}$ is projectively compact. \qed
\end{prop}

The goal of the present paper is to identify the closure $\overline{\mathcal{O}(G,\mathcal{F})}$ with the space of \emph{very small} $(G,\mathcal{F})$-trees, which are defined in the following way. 

\begin{de}
A $(G,\mathcal{F})$-tree $T$ is \emph{small} if arc stabilizers in $T$ are either trivial, or cyclic (and non-peripheral). A $(G,\mathcal{F})$-tree $T$ is \emph{very small} if it is small, and in addition nontrivial arc stabilizers in $T$ are closed under taking roots, and tripod stabilizers in $T$ are trivial.
\end{de}

We note that the trivial action of $G$ on a point is very small in the above sense. We denote by $VSL(G,\mathcal{F})$ the subspace of $\mathbb{PR}^{G}$ made of projective classes of minimal, nontrivial, very small $(G,\mathcal{F})$-trees, which we equip with the axes topology.

\subsection{The equivariant Gromov--Hausdorff topology}\label{sec-Lipschitz}

\paragraph*{The equivariant Gromov--Hausdorff topology on the space of $(G,\mathcal{F})$-trees.} The space $\mathcal{O}(G,\mathcal{F})$ can also be equipped with the \emph{equivariant Gromov--Hausdorff topology} \cite{Pau88}, which is equivalent to the axes topology \cite{Pau89}. We now recall the definition of the equivariant Gromov--Hausdorff topology on the space of $(G,\mathcal{F})$-trees. Let $T$ and $T'$ be two $(G,\mathcal{F})$-trees, let $K\subset T$ and $K'\subset T'$ be finite subsets, let $P\subset G$ be a finite subset of $G$, and let $\epsilon>0$. A \emph{$P$-equivariant $\epsilon$-relation} between $K$ and $K'$ is a subset $R\subseteq K\times K'$ whose projection to each factor is surjective, such that for all $(x,x'),(y,y')\in R$ and all $g,h\in P$, we have $|d_T(gx,hy)-d_{T'}(gx',hy')|<\epsilon$. A basis of open sets for the equivariant Gromov--Hausdorff topology is given by the sets $O(T,K,P,\epsilon)$ of all $(G,\mathcal{F})$-trees $T'$ for which there exist a finite subset $K'\subset T'$ and a $P$-equivariant $\epsilon$-relation $R\subseteq K\times K'$ \cite{Pau88}.

\paragraph*{The equivariant Gromov--Hausdorff topology on the space of pointed $(G,\mathcal{F})$-trees.} The equivariant Gromov--Hausdorff topology can also be defined on the space of pointed $(G,\mathcal{F})$-trees. Let $T$ be a $(G,\mathcal{F})$-tree, and let $(x_1,\dots,x_l)\in T^l$. Let $K\subset T$ and $P\subset G$ be finite subsets, and let $\epsilon>0$. A basis of open sets for the equivariant Gromov--Hausdorff topology is given by the sets $O'((T,(x_1,\dots,x_l)),K,P,\epsilon)$ of all pointed $(G,\mathcal{F})$-trees $(T',(x'_1,\dots,x'_l))$ for which there exist a finite subset $K'\subset T'$ and a $P$-equivariant $\epsilon$-relation $R\subseteq (K\cup\{x_1,\dots,x_l\})\times (K'\cup\{x'_1,\dots,x'_l\})$ with $(x_i,x'_i)\in R$ for all $i\in \{1,\dots,l\}$. 

Let $T$ be a $(G,\mathcal{F})$-tree, let $x\in T$, and let $(T_n)_{n\in\mathbb{N}}$ be a sequence of $(G,\mathcal{F})$-trees that converges to $T$ in the equivariant Gromov--Hausdorff topology. A sequence $(x_n)_{n\in\mathbb{N}}\in \prod_{n\in\mathbb{N}}T_n$ is an \emph{approximation} of $x$ if the sequence $((T_n,x_n))_{n\in\mathbb{N}}$ of pointed $(G,\mathcal{F})$-trees converges to $(T,x)$. 

\begin{prop} (Horbez \cite[Theorem 4.3]{Hor14-1})\label{pointed-convergence}
Let $(T,u)$ (resp. $(T',u')$) be a pointed very small $(G,\mathcal{F})$-tree, and let $((T_n,u_n))_{n\in\mathbb{N}}$ (resp. $((T'_n,u'_n))_{n\in\mathbb{N}}$) be a sequence of pointed very small $(G,\mathcal{F})$-trees that converges to $(T,u)$ (resp. $(T',u')$) in the equivariant Gromov--Hausdorff topology. Assume that for all $n\in\mathbb{N}$, there exists a $1$-Lipschitz $G$-equivariant map $f_n:T_n\to T'_n$, such that $f_n(u_n)=u'_n$. Then there exists a $1$-Lipschitz $G$-equivariant map $f:T\to\overline{T'}$, such that $f(u)=u'$, where $\overline{T'}$ denotes the metric completion of $T'$.
\end{prop}

\subsection{Graphs of actions and transverse coverings}

Let $G$ be a countable group, and $\mathcal{F}$ be a free factor system of $G$. A \emph{$(G,\mathcal{F})$-graph of actions} consists of

\begin{itemize}
\item a marked graph of groups $\mathcal{G}$, whose fundamental group is isomorphic to $G$, such that all subgroups in $\mathcal{F}$ are conjugate into vertex groups of $\mathcal{G}$, together with an assignment of length (possibly equal to $0$) to each edge of $\mathcal{G}$, in such a way that every edge with peripheral stabilizer has length $0$, and 

\item an isometric action of every vertex group $G_v$ on a $G_v$-tree $T_v$ (possibly reduced to a point), in which all intersections of $G_v$ with peripheral subgroups of $G$ are elliptic and fix a unique point, and

\item a point $p_e\in T_{t(e)}$ fixed by $i_e(G_e)\subseteq G_{t(e)}$ for every oriented edge $e$.
\end{itemize}

A $(G,\mathcal{F})$-graph of actions is \emph{nontrivial} if the associated graph of groups is not reduced to a point. Associated to any $(G,\mathcal{F})$-graph of actions $\mathcal{G}$ is a $G$-tree $T(\mathcal{G})$. Informally, the tree $T(\mathcal{G})$ is obtained from the Bass--Serre tree $S$ of the underlying graph of groups by equivariantly attaching each vertex tree $T_v$ at the corresponding vertex $v$, an incoming edge being attached to $T_v$ at the prescribed attaching point, and making each edge from $S$ whose assigned length is equal to $d$ isometric to the segment $[0,d]$ (edges of length $0$ are collapsed). The reader is referred to \cite[Proposition 3.1]{Gui98} for a precise description of the tree $T(\mathcal{G})$. We say that a $(G,\mathcal{F})$-tree $T$ \emph{splits as a $(G,\mathcal{F})$-graph of actions} if there exists a $(G,\mathcal{F})$-graph of actions $\mathcal{G}$ such that $T=T({\mathcal{G}})$.

A \emph{transverse covering} of an $\mathbb{R}$-tree $T$ is a family $\mathcal{Y}$ of nondegenerate closed subtrees of $T$ such that every arc in $T$ is covered by finitely many subtrees in $\mathcal{Y}$, and for all $Y\neq Y'\in\mathcal{Y}$, the intersection $Y\cap Y'$ contains at most one point. It is \emph{trivial} if $\mathcal{Y}=\{T\}$, and \emph{nontrivial} otherwise. The \emph{skeleton} of $\mathcal{Y}$ is the simplicial tree $S$ defined as follows. The vertex set of $S$ is the set $\mathcal{Y}\cup V_0(S)$, where $V_0(S)$ is the set of all intersection points between distinct subtrees in $\mathcal{Y}$. There is an edge between $Y\in\mathcal{Y}$ and $y\in V_0(S)$ whenever $y\in Y$. 

\begin{prop} (Guirardel \cite[Lemma 1.5]{Gui08})\label{skeleton}
A $(G,\mathcal{F})$-tree splits as a nontrivial $(G,\mathcal{F})$-graph of actions if and only if it admits a nontrivial $G$-invariant transverse covering.
\end{prop}

\section{Closedness of the space of projective very small $(G,\mathcal{F})$-trees}\label{sec-compactness}

We denote by $VSL(G,\mathcal{F})$ the space of projective very small, minimal, nontrivial $(G,\mathcal{F})$-trees, as defined in Section \ref{sec-outer-space}. The goal of the present section is to establish the following fact. 

\begin{prop}\label{very-small}
The space $VSL(G,\mathcal{F})$ is closed in $\mathbb{PR}^G$.
\end{prop}

In other words, every limit point of a sequence of very small $(G,\mathcal{F})$-trees is very small. This was proved by Cohen--Lustig for actions of finitely generated groups on $\mathbb{R}$-trees \cite[Theorem I]{CL95}: by working in the axes topology, they proved closedness of the conditions that nontrivial arc stabilizers are cyclic and root-closed, and that tripod stabilizers are trivial. We will provide a shorter proof of these facts by working in the equivariant Gromov--Hausdorff topology. A $(G,\mathcal{F})$-tree $T$ is \emph{irreducible} if no end of $T$ is fixed by all elements of $G$. The Gromov--Hausdorff topology is equivalent to the axes topology on the space of minimal, irreducible $(G,\mathcal{F})$-trees \cite{Pau89}. Since limits of trees in $\mathcal{O}(G,\mathcal{F})$ are irreducible, we can carry our arguments in the equivariant Gromov--Hausdorff topology. A proof of the fact that being small is a closed condition (in the equivariant Gromov--Hausdorff topology) also appears in \cite[Lemme 5.7]{Pau88}. In our setting, we also need to check closedness of the condition that arc stabilizers are not peripheral. We will make use of classical theory of group actions on $\mathbb{R}$-trees, and refer the reader to \cite{CM87} for an introduction to this theory.

\begin{lemma}\label{lemma-limit-good}
Let $T$ be a minimal $(G,\mathcal{F})$-tree, and let $(T_n)_{n\in\mathbb{N}}$ be a sequence of minimal $(G,\mathcal{F})$-trees that converges (non-projectively) to $T$. Let $g\in G$ be an element that fixes a nondegenerate arc in $T$. Then for all $n\in\mathbb{N}$ sufficiently large, either $g$ fixes a nondegenerate arc in $T_n$, or $g$ is hyperbolic in $T_n$.
\end{lemma}

\begin{proof}
Otherwise, up to passing to a subsequence, we can assume that for all $n\in\mathbb{N}$, the element $g$ fixes a single point $x_n$ in $T_n$. Let $[a,b]$ be a nondegenerate arc fixed by $g$ in $T$. Let $a_n$ (resp. $b_n$) be an approximation of $a$ (resp. $b$) in the tree $T_n$. As $d_{T_n}(a_n,ga_n)$ and $d_{T_n}(b_n,gb_n)$ both tend to $0$, the points $a_n$ and $b_n$ are both arbitrarily close to $x_n$. Therefore, the distance $d_{T_n}(a_n,b_n)$ converges to $0$, and $a=b$, a contradiction.
\end{proof}

\begin{lemma}\label{lemma-good-2}
Let $T$ be a minimal $(G,\mathcal{F})$-tree, and let $(T_n)_{n\in\mathbb{N}}$ be a sequence of minimal $(G,\mathcal{F})$-trees that converges to $T$. Let $g\in G$. Assume that some power $g^k$ of $g$ fixes a nondegenerate arc $I$ in $T$. If for all sufficiently large $n\in\mathbb{N}$, the element $g$ is hyperbolic in $T_n$, then $g$ fixes $I$. 
\end{lemma}

\begin{proof}
Let $I:=[a,b]$. Let $a_n$ (resp. $b_n$) be an approximation of $a$ (resp. $b$) in $T_n$. Since $d_{T_n}(a_n,g^ka_n)$ and $d_{T_n}(b_n,g^kb_n)$ both converge to $0$, the points $a_n$ and $b_n$ are arbitrarily close to the axis of $g$ in $T_n$, and $||g||_{T_n}$ converges to $0$. Hence both $d_{T_n}(a_n,ga_n)$ and $d_{T_n}(b_n,gb_n)$ converge to $0$, so $g$ fixes $[a,b]$.
\end{proof}

\begin{proof}[Proof of Proposition \ref{very-small}]
Let $(T_n)_{n\in\mathbb{N}}$ be a sequence of very small $(G,\mathcal{F})$-trees that converges to a $(G,\mathcal{F})$-tree $T$. Let $g\in G$ be a peripheral element. Then for all $n\in\mathbb{N}$, the element $g$ fixes a single point in $T_n$. Lemma \ref{lemma-limit-good} implies that $g$ fixes a single point in $T$.
\\
\\
\indent Let now $g,h\in G$ be two elements that fix a common nondegenerate arc $[a,b]\subseteq T$. We will show that the group $\langle g,h\rangle$ is abelian, and hence cyclic because $g$ and $h$ are nonperipheral. Let $a_n$ (resp. $b_n$) be an approximation of $a$ (resp. $b$) in $T_n$. Let $\epsilon>0$, chosen to be small compared to $d_T(a,b)$. Since $d_{T_n}(a_n,ga_n),d_{T_n}(b_n,gb_n)\le\epsilon$ for $n$ large enough, while $d_{T_n}(a_n,b_n)\ge d_T(a,b)-\epsilon$, the elements $g$ and $h$ are hyperbolic in $T_n$, and their translation axes have an overlap of length greater than $3\epsilon$. On the other hand, we have $||g||_{T_n},||h||_{T_n}\le\epsilon$. This implies that the elements $[g,h]$, $g[g,h]g^{-1}$ and $h[g,h]h^{-1}$ all fix a common nondegenerate arc in $T_n$. As $T_n$ is very small, the group generated by these elements is (at most) cyclic, and in addition $[g,h]$ is nonperipheral. This implies that $[g,h]$ is hyperbolic in any Grushko $(G,\mathcal{F})$-tree. Both $g$ and $h$ preserve the axis of $[g,h]$ in a Grushko $(G,\mathcal{F})$-tree, and hence $g$ and $h$ commute. 
\\
\\
\indent Let now $g\in G$ be an element, one of whose proper powers $g^k$ fixes a nondegenerate arc $[a,b]\subseteq T$. 

We first assume that $g$ fixes a nondegenerate arc in $T_n$ for all $n\in\mathbb{N}$, and let $I_n$ denote the fixed point set of $g$ in $T_n$. Since $T_n$ is very small, the element $g$ also fixes $I_n$ for all $n\in\mathbb{N}$. Let $a_n$ (resp. $b_n$) be an approximation of $a$ (resp. $b$) in $T_n$. Since $d_{T_n}(g^ka_n,a_n)$ and $d_{T_n}(g^kb_n,b_n)$ both converge to $0$, the arc $I_n$ comes arbitrarily close to both $a_n$ and $b_n$. This implies that both $d_{T_n}(ga_n,a_n)$ and $d_{T_n}(gb_n,b_n)$ converge to $0$, and therefore $g$ fixes $[a,b]$ in $T$. 

Otherwise, up to passing to a subsequence, we can assume that $g^k$, and hence $g$, is hyperbolic in $T_n$ for all $n\in\mathbb{N}$. It then follows from Lemma \ref{lemma-good-2} that $g$ fixes $[a,b]$. 
\\
\\
\indent We finally assume that $g$ fixes a nondegenerate tripod in $T$, whose extremities we denote by $a,b$ and $c$. Let $m$ be the center of this tripod, and $L>0$ be the shortest distance in $T$ between $m$ and one of the points $a$, $b$ or $c$. Let $a_n$ (resp. $b_n,c_n,m_n$) be an approximation of $a$ (resp. $b,c,m$) in $T_n$, and let $\epsilon>0$ be such that $\epsilon<\frac{L}{2}$. For $n$ sufficiently large, the point $m_n$ lies at distance at most $\epsilon$ from the center $m'_n$ of the tripod formed by $a_n$, $b_n$ and $c_n$ in $T_n$. In addition, as $a_n$, $b_n$ and $c_n$ all lie at distance at most $\epsilon$ from $C_{T_n}(g)$, the distance from $m'_n$ to one of the points $a_n$, $b_n$ or $c_n$ is at most $\epsilon$. This leads to a contradiction.
\end{proof}

\section{Geometric $(G,\mathcal{F})$-trees}\label{sec-geom}

In the present section, we introduce the class of \emph{geometric $(G,\mathcal{F})$-trees}, that will be an important tool in our analysis of the closure of $\mathcal{O}(G,\mathcal{F})$. In particular, we will explain how to approximate any very small $(G,\mathcal{F})$-tree by geometric ones. Our presentation was largely inspired by Gaboriau--Levitt's work in the context of free groups \cite{GL95}. The idea of using band complexes and systems of isometries for studying trees dates back to the work of Rips. 

\subsection{Approximations by geometric $(G,\mathcal{F})$-trees} \label{sec-pg-tree}

Given a tree $T$, we define a \emph{finite tree} in $T$ as the convex hull of a finite collection of points in $T$. A \emph{finite forest} of $T$ is a subset of $T$ having finitely many connected components, each of which is a finite tree.

Let $B:=\{g_1,\dots,g_N\}$ be a fixed free basis of $F_N$, and let $S$ be the standard $(G,\mathcal{F})$-free splitting represented on Figure \ref{fig-Grushko-rose}. Let $T$ be a $(G,\mathcal{F})$-tree, and let $K\subseteq T$ be a finite forest. Let $v_0\in S$ be a vertex with trivial stabilizer, and let $v_1,\dots,v_k$ be the adjacent vertices with stabilizers $G_1,\dots,G_k$. Notice that the set $\{v_0,\dots,v_k\}$ is a set of representatives of all $G$-orbits of vertices in $S$. 

Given a vertex $v\in V(S)$ in the orbit of $v_i$ (for some $i\in\{0,\dots,k\}$), we let $K_v$ be the union, taken over all elements $g\in G$ such that $gv_i=v$, of the translates $gK$. Notice in particular that for all $i\in\{1,\dots,k\}$, we have $K_{v_i}=G_i.K$, and in general the forest $K_v$ is $\text{Stab}(v)$-invariant. Let $$X:=\coprod_{v\in V(S)}K_v.$$ Let $\Sigma$ be the foliated complex obtained from $X$ by adding, for each pair $\{v,v'\}$ of adjacent vertices in $S$, a band $K_v\cap K_{v'}\times [0,1]$, joining the two copies of $K_v\cap K_{v'}$ sitting in $K_v$ and in $K_{v'}$. These bands are foliated by the vertical sets of the form $\{x\}\times [0,1]$. 

Notice that $\Sigma$ can also be viewed as the subset of $T\times S$ made of all couples of the form $(x,y)$, where either $y$ is a vertex of $S$ and $x\in K_y$, or else $y$ belong to the interior of an edge joining two vertices $v$ and $v'$ in $S$, and $x\in K_v\cap K_{v'}$. The diagonal action of $G$ on $T\times S$ then restricts to a $G$-action on $\Sigma$. The projections to $T$ and $S$ define natural $G$-equivariant maps $\pi_T:\Sigma\to T$ and $\pi_S:\Sigma\to S$. We will let $\Sigma_0:=\pi_S^{-1}(G.v_0)$

We now assume that the set $K\subseteq T$ is a finite tree, i.e. $K$ is connected, that $K\cap gK\neq\emptyset$ for all $g\in B$, and that for all $i\in\{1,\dots,k\}$, the tree $K$ contains the unique point $x_i$ which is fixed by $G_i$. In this way, all subsets $K_v$ with $v\in V(S)$ are connected. Then $\Sigma$ projects onto the tree $S$, and all fibers of the projection map $\pi_S$ are trees, which implies that $\Sigma$ is contractible. It follows from \cite[Proposition 1.7]{LP97} that the leaf space of $\Sigma$ made Hausdorff is an $\mathbb{R}$-tree $T_K$. Notice by construction that for all vertices $v_i$ with nontrivial stabilizer in $S$, the tree $K_{v_i}$ is equal to the union of all trees $K_v$ with $v$ adjacent to $v_i$. Therefore, every point in $T_K$ has a representative in $\Sigma_0$.

We will now establish a few properties of $T_K$, thus extending a theorem of Gaboriau and Levitt \cite[Theorem I.1]{GL95} to the context of $(G,\mathcal{F})$-trees. Given two $\mathbb{R}$-trees $T$ and $T'$, a \emph{morphism} $f:T\to T'$ is a map such that every segment in $T$ can be subdivided into finitely many subsegments, in restriction to which $f$ is an isometry.

\begin{theo} \label{pseudogroup-tree}
Let $T$ be a $(G,\mathcal{F})$-tree, and let $K\subseteq T$ be a finite tree that contains the fixed point of $G_i$ for all $i\in\{1,\dots,k\}$, and such that $K\cap gK\neq\emptyset$ for all $g\in B$. Then $T_K$ is the unique $(G,\mathcal{F})$-tree such that
\begin{enumerate}
\item the finite tree $K$ embeds isometrically into $T_{K}$, and 
\item for all $g\in G$ and all $x,y\in K$, if $g.x=y\in T$ and the points $(y,v_0)$ and $(y,gv_0)$ belong to the same leaf of $\Sigma$, then $g.x=y$ in $T_K$, and 
\item every segment of $T_{K}$ is contained in a finite union of translates $g.K$ with $g\in G$, and
\item if $T'$ is any $(G,\mathcal{F})$-tree satisfying the first two above properties, then there exists a unique $G$-equivariant morphism $j:T_{K}\to T'$ such that $j(x)=x$ for all $x\in K$.
\end{enumerate}
\end{theo}

\begin{proof}
The proof of Theorem \ref{pseudogroup-tree} is the same as the proof of \cite[Theorem I.1]{GL95}, we will sketch a proof for completeness. Uniqueness follows from the fourth property: if $T'$ is another $(G,\mathcal{F})$-tree with the same properties, then there are $G$-equivariant morphisms from $T_K$ to $T'$ and from $T'$ to $T_K$. Since morphisms can only decrease translation length functions, this implies that $T_K$ and $T'$ are $G$-equivariantly isometric. 

The tree $T_K$ is obtained in the following way. Given any two points $x,y\in\Sigma_0$, we let $$\delta(x,y):=\inf_{\sigma} d_T(x,u_1)+d_T(u'_1,u_2)\dots+d_T(u'_{k-1},y),$$ where the infimum is taken over the set of all sequences $$\sigma=(x=u'_0,u_1,u'_1,u_2,u'_2,\dots,u_{k-1},u'_{k-1},u_k=y)$$ of points in $\Sigma$, such that for all $i\in\{0,\dots,k-1\}$, the segment $[u'_i,u_{i+1}]$ is horizontal (i.e. its projection to $S$ is constant), and for all $i\in\{1,\dots,k\}$, the segment $[u_i,u'_i]$ is vertical (i.e. its projection to $T$ is constant), and hence $[u'_i,u_{i+1}]$ is contained in a leaf of $\Sigma$. Then $\delta$ defines a pseudo-metric on $\Sigma_0$, and $T_K$ is the metric space $(\Sigma_0,\delta)$ made Hausdorff. 

Property $2$ in Theorem \ref{pseudogroup-tree} easily follows from the construction, and Property $3$ follows from Properties $1$ and $2$. Property $4$ is obtained by noticing that if $T'$ is a tree satisfying the first two conclusions of the theorem, then if we let $\delta'((x,gv_0),(x',g'v_0)):=d_{T'}(gx,g'x')$ for all $x,x'\in K$ and all $g,g'\in G$, then we have $\delta'\le\delta$. Then using the fact that every segment in $T'$ is covered by finitely many $G$-translates of $K$, we get that the identity map on $K\times\Sigma_0$ induces the required morphism from $T'$ to $T$. We refer to the proof of \cite[Theorem I.1]{GL95} for details of the arguments.

The key point for establishing Property $1$ will be to prove that if $\delta(x,y)=0$, then $x$ and $y$ belong to the same leaf of the foliation. We will prove more generally that the infimum in the definition of $\delta$ is achieved, and that the $\pi_S$-images of the points in the minimizing sequence are aligned in the right order in $S$. Property $1$ in Theorem \ref{pseudogroup-tree} will follow from this fact: if $x,y\in K$, then any minimizing sequence has to be contained in $K$, and therefore $\delta(x,y)=d_T(x,y)$. 

We will now prove that the infimum in the definition of $\delta$ is achieved. Given a sequence $\sigma$ as above, we let $$\delta(\sigma):= d_T(x,u_1)+d_T(u'_1,u_2)+\dots+d_T(u'_{k-1},y).$$ In order to prove the above fact, we first notice that if the $\pi_S$-image of the sequence $\sigma$ backtracks, then one can find another sequence $\sigma'$ with $\delta(\sigma')\le\delta(\sigma)$: indeed, if $\sigma$ contains a subsequence of the form $u_i,u'_{i},u_{i+1},u'_{i+1}$ where $\pi_S(u_i)=\pi_S(u'_{i+1}):=v$, then this subsequence can be replaced by $u_i,u'_{i+1}$ (indeed, the segment $[u_i,u'_{i+1}]$ is then contained in $K_v$), and this yields a new sequence $\sigma'$ satisfying $\delta(\sigma')\le\delta(\sigma)$. 

In addition, let $u'_i,u_{i+1},u'_{i+1}$ be a finite subsequence arising in a sequence $\sigma$ as above. Let $\gamma:=\pi_S([u_{i+1},u'_{i+1}])$. Let $\gamma_K\subseteq K$ be the subset of $K$ made of those points $x\in K$ such that there exists a leaf segment in $\Sigma$ containing $x$ and projecting to $\gamma$. Then one does not increase the value of $\delta(\sigma)$ if one replaces $u_{i+1}$ by the projection $y$ of $u_i$ to $\gamma_K$, and $u'_{i+1}$ by the unique point $z$ in $\Sigma$ in the same leaf as $y$, such that $\pi_S([y,z])=\gamma$. This implies that the supremum can be taken over the finite set made of those points in $K$ that are extremal in some band of $\Sigma$. In particular, this supremum is achieved.  
\end{proof}

If $T'$ is a tree satisfying the first two conclusions of Theorem \ref{pseudogroup-tree}, the morphism $j$ provided by Theorem \ref{pseudogroup-tree} is called a \emph{resolution} of $T'$. 

\begin{de}
A $(G,\mathcal{F})$-tree $T$ is \emph{geometric} if there exists a finite subtree $K\subseteq T$ such that $T=T_{K}$.
\end{de}

Let $T$ be a $(G,\mathcal{F})$-tree, and let $(T_n)_{n\in\mathbb{N}}$ be a sequence of $(G,\mathcal{F})$-trees. The sequence $(T_n)_{n\in\mathbb{N}}$ \emph{strongly converges} towards $T$ (in the sense of Gillet and Shalen \cite{GS90}) if for all integers $n\le n'$, there exist morphisms $j_{n,n'}:T_n\to T_{n'}$ and $j_n:T_n\to T$ such that for all $n\le n'$ and all segments $I\subseteq T_n$, the morphism $j_{n'}$ is an isometry in restriction to $j_{n,n'}(I)$. Strong convergence implies in particular that for all $g\in G$, there exists $n_0\in\mathbb{N}$ such that for all $n\ge n_0$, we have $||g||_{T_{n}}=||g||_T$. The following theorem essentially follows from work by Levitt and Paulin \cite[Theorem 2.2]{LP97} and Gaboriau and Levitt \cite[Proposition II.1]{GL95}. 

\begin{theo}\label{approx-by-geom}
Let $G$ be a countable group, and let $\mathcal{F}$ be a free factor system of $G$. Let $T$ be a minimal $(G,\mathcal{F})$-tree. Then there exists a sequence $(T_n)_{n\in\mathbb{N}}$ of minimal geometric $(G,\mathcal{F})$-trees that strongly converges towards $T$. If in addition $H$ is a subgroup of $G$ with finite Kurosh rank that is elliptic in $T$, then the approximation can be chosen so that $H$ is elliptic in $T_n$.
\end{theo}

\begin{proof}
By carefully choosing an increasing sequence of finite subtrees $K_n\subseteq T$ whose union is the whole tree $T$, we will establish that the trees $T_{K_n}$ strongly converge towards $T$. One point in the proof is to carefully choose $K_n$ so that all trees $T_{K_n}$ are minimal: the idea for this is to choose a subtree $K_n$ all of whose endpoints belong to the $G$-orbit of some fixed basepoint $x_0\in T$. Here are the details of the proof.

Let $T$ be a minimal $(G,\mathcal{F})$-tree. Let $(g_n)_{n\in\mathbb{N}}$ be an enumeration of $G$. Assume that this enumeration is chosen so that if $v_0$, $g_nv_0$ and $g_mv_0$ are aligned in this order in $S$, then $m\ge n$. Let $x_0\in T$, and for all $n\ge 1$, let $K_n$ be the convex hull of $\{g_kx_0|k\le n\}$ in $T$, so that by construction, all extreme points of $K_n$ belong to the $G$-orbit of $x_0$. By minimality, the tree $T$ is the increasing union of the trees $K_n$, and there exists $g\in G$ such that $x_0$ belongs to the translation axis of $g$ in $T$. Let $n_0\in\mathbb{N}$ be large enough so that $K_{n_0}$ contains $gx_0$ and satisfies the hypotheses from Theorem \ref{pseudogroup-tree}.

For all $n\ge n_0$, we let $T_n:=T_{K_n}$. The first property in Theorem \ref{pseudogroup-tree} implies that the distance between $x_0$ and $gx_0$ is the same in $T_n$ and in $T$, so $x_0$ belongs to the axis of $g$ in $T_n$. In addition, the second property in Theorem \ref{pseudogroup-tree}, together with our choice of enumeration of $G$, implies that $T_n$ is the convex hull of the orbit of $x_0$. Hence $T_n$ is minimal. 

We claim that the sequence $(T_n)_{n\in\mathbb{N}}$ strongly converges towards $T$. Given any two integers $n\le n'$, we have $K_n\subseteq K_{n'}$. Theorem \ref{pseudogroup-tree} applied to $K_{n'}$ implies that $T_{n'}$ contains an isometrically embedded copy of $K_{n'}$, and hence of $K_n$. In addition, if two points belong to the same leaf of $\Sigma_n$, then they also belong to the same leaf of $\Sigma_{n'}$ (where $\Sigma_n$ and $\Sigma_{n'}$ denote the band complexes associated to $K_n$ and $K_{n'}$, respectively). Therefore $T_{n'}$ satisfies the first two properties of Theorem \ref{pseudogroup-tree} with respect to the finite tree $K_n$. The last property in Theorem \ref{pseudogroup-tree} (applied to $K_n$) then provides morphisms $j_{n,n'}: T_n\to T_{n'}$ for all $n\le n'$, as well as morphisms $j_n:T_n\to T$ for all $n\in\mathbb{N}$. Let $n\in\mathbb{N}$, and $I\subseteq T_n$. The third property in Theorem \ref{pseudogroup-tree} then enables us to choose a finite set $Y$ of elements of $G$ so that $I$ is covered by the translates of $K_n$ in $T_n$ under elements in $Y$. We then let $n'$ be large enough so that $K_{n'}$ contains the translates of $K_n$ in $T$ under the elements in $Y$. Then $j_{n'}$ is an isometry from $j_{n,n'}(I)$ to $j_n(I)$. This shows strong convergence of the trees $T_n$ towards $T$.  

By choosing $n$ large enough so that all elements in a basis of the free part of the Kurosh decomposition of $H$, as well as all conjugators arising in this decomposition, are of the form $g_k$ for some $k\le n$, and $K_n$ contains a fixed point of $H$, we can ensure that the last property of Theorem \ref{approx-by-geom} is satisfied.
\end{proof}

\begin{rk}\label{rk-vol}
Notice that if branch points are dense in $T$, then the tree $K_n$ has edges of arbitrarily small length as $n$ tends to $+\infty$, from which it follows that simplicial edges in the geometric approximation of $T$ have lengths going to $0$. 
\end{rk}

\subsection{Properties of geometric $(G,\mathcal{F})$-trees}\label{sec-stab}

We now list a few other useful properties of the tree $T_{K}$, which were proved by Gaboriau and Levitt in \cite{GL95} in the case of $F_N$-trees. Any element of $G$ can be written in a unique way as a \emph{reduced} word of the form $s_1\dots s_k$, where
\begin{itemize}
\item for all $i\in\{1,\dots,k\}$, either $s_i\in B\cup B^{-1}$, or else $s_i$ belongs to one of the peripheral groups $G_j$, and
\item no two consecutive letters are of the form $ss^{-1}$ or $s^{-1}s$ with $s\in B$, and
\item no two consecutive letters belong to the same peripheral group $G_j$. 
\end{itemize}
\noindent It is \emph{cyclically reduced} if in addition, it does not start and end with letters that are inverses to each other, and does not either start and end with letters belonging to a common peripheral group.

\begin{prop}\label{prop-abg}
Let $T$ be a $(G,\mathcal{F})$-tree, let $K\subseteq T$ be a finite subtree that contains the fixed points of all groups $G_j$, and such that $K\cap sK\neq\emptyset$ for all $s\in B$, and let $g\in G$ be nonperipheral and cyclically reduced. Then the fixed point set of $g$ in $T_{K}$ is contained in $K$.
\end{prop}

\begin{proof}
The proof goes as in \cite[Proposition I.5]{GL95}. Let $a\in T_{K}$ be a fixed point of $g$. Choose a representative $(x,v)\in \Sigma_0$ of $a$, with $v=hv_0$ for some $h\in G$, such that the distance from $v_0$ to $v$ is minimal. We let $x_0:=h^{-1}x$. Assume towards a contradiction that $v\neq v_0$, i.e. $h\neq e$. We have $(x,v)=(gx,gv)=(x,gv)$ in $T_K$: the first equality comes from the fact that $g(x,v)=(x,v)$ in $T_K$, and the second from the fact that $gx=x$ in $T$. By applying $h^{-1}$, we get that $(x_0,v_0)=(x_0,h^{-1}ghv_0)$. It follows from Theorem~\ref{pseudogroup-tree} that $(x_0,v_0)$ and $(x_0,h^{-1}ghv_0)$ lie in a common leaf of $\Sigma$. Let $g_1$ (resp. $g_k$) be the first (resp. last) letter in the reduced word representing $h^{-1}gh$. Then there is a leaf in $\Sigma$ joining $(x_0,v_0)$ to $(x_0,g_1v_0)$, and there is a leaf in $\Sigma$ joining $(x_0,h^{-1}ghg_k^{-1}v_0)$ to $(x_0,h^{-1}ghv_0)$. By translating those by $h$ and $g^{-1}h$, respectively, we find leaves in $\Sigma$ joining $(x,v)$ to $(x,hg_1v_0)$ and to $(x,hg_k^{-1}v_0)$. Our choice of $v$ then implies that the reduced word representing $h$ cannot end with $g_k$, and the reduced word representing $h^{-1}$ cannot begin with $g_1$. But since $h^{-1}gh$ begins with $g_1$ and ends with $g_k$, this forces $g$ not to be cyclically reduced, a contradiction.
\end{proof} 

\begin{cor}\label{stab-isom}
Let $T$ be a small $(G,\mathcal{F})$-tree, and let $K$ be a finite subtree of $T$ that contains all fixed points of the $G_i$'s, and such that $K\cap sK\neq\emptyset$ for all $s\in B$. For all $g\neq 1\in G$, the restriction of the resolution map $j:T_{K}\to T$ to the fixed point set of $g$ is an isometry. In particular, arc stabilizers in $T_{K}$ are either trivial, or cyclic and non-peripheral. If $T$ is very small, then tripod stabilizers in $T_{K}$ are trivial.
\qed
\end{cor}  

Let $k\in\mathbb{N}$. A $(G,\mathcal{F})$-tree $T$ is \emph{$k$-tame} if $T$ is small, and in addition, we have $\text{Fix}(g^{kl})=\text{Fix}(g^k)$ for all $l\ge 1$. We refer to Section \ref{sec-good} for details and equivalent definitions.

\begin{cor}\label{lemma-abg}
Let $T$ be a small $(G,\mathcal{F})$-tree, and let $K$ be a finite subtree of $T$ that contains all fixed points of the $G_i$'s, and such that $K\cap sK\neq\emptyset$ for all $s\in B$. For all $k\in\mathbb{N}$, if $T$ is $k$-tame, then $T_{K}$ is $k$-tame.
\end{cor}

\begin{proof}
Let $g\in G$ be nonperipheral, and let $k\in\mathbb{N}$. Up to passing to a conjugate, we can assume that $g$ is cyclically reduced, and in this case the fixed point set of $g$ in $T_{K}$ is contained in $K$ by Proposition \ref{prop-abg}. Let $x\in K$ be such that there exists $l\in\mathbb{N}$ such that $g^{kl}x=x$ in $T_K$. This implies in particular that $g^{kl}x=x$ in $T$. We will show that $g^{k}x=x$ in $T_K$. By Theorem \ref{pseudogroup-tree}, there is a leaf segment in $\Sigma$ joining $(x,v_0)$ to $(x,g^{kl}v_0)$. But since $g$ is cyclically reduced, this implies that there is also a leaf segment in $\Sigma$ (which is a subsegment of the previous one) joining $(x,v_0)$ to $(x,g^kv_0)$. Since $g^kx=x$ in $T$ (because $T$ is $k$-tame), we have found a leaf segment in $\Sigma$ joining $(x,v_0)$ to $(g^kx,g^kv_0)$. This implies that $g^kx=x$ in $T_K$, as required. 
\end{proof}

As a consequence of Theorem \ref{approx-by-geom} and the above two corollaries, we get the following approximation result. 

\begin{theo}\label{approx-by-geom-2}
Let $G$ be a countable group, and let $\mathcal{F}$ be a free factor system of $G$. Let $T$ be a minimal, small $(G,\mathcal{F})$-tree. Then there exists a sequence $(T_n)_{n\in\mathbb{N}}$ of minimal, small, geometric $(G,\mathcal{F})$-trees, that strongly converges towards $T$. If $T$ is very small (resp. $k$-tame for some $k\in\mathbb{N}$), the approximation can be chosen very small (resp. $k$-tame).
\qed
\end{theo}

\subsection{Interpretation in terms of systems of isometries}\label{sec-systems}

Let $K$ be a finite forest. Given two nonempty closed subtrees $A$ and $B$ of $K$, a \emph{morphism} from $A$ to $B$ is a surjective map $\phi:A\to B$, such that $A$ can be subdivided into finitely many subtrees in restriction to which $\phi$ is an isometry. The subtrees $A$ and $B$ are called the \emph{bases} of $\phi$. A \emph{finite system of morphisms} is a pair $\mathcal{K}=(K,\Phi)$, where $K$ is a finite forest, and $\Phi$ is a finite collection of morphisms between nonempty closed subtrees of $K$.

In the case where all morphisms $\phi\in\Phi$ are isometries between their bases, we call $\mathcal{K}$ a \emph{finite system of isometries}. Notice that any finite system of morphisms can be turned into a finite system of isometries by appropriately subdividing the bases of the morphisms in $\Phi$.
 
We now assume that $\Phi$ is a finite system of isometries. Given an isometry $\phi$ from $A_{\phi}$ to $B_{\phi}$, we denote by $\phi^{-1}$ its inverse, which is a partial isometry from $B_{\phi}$ to $A_{\phi}$. Given partial isometries $\phi_1,\dots,\phi_n$, we denote by $\phi_1\circ\dots\circ\phi_n$ the composition of the $\phi_i$'s, which is a partial isometry whose domain is the set of all $x\in A_{\phi_n}$ such that for all $i\in\{2,\dots,n\}$, we have $\phi_i\circ\dots\circ\phi_n(x)\in A_{\phi_{i-1}}$. A word in the partial isometries in the family $\Phi$ and their inverses is \emph{reduced} if it does not contain any subword of the form $\phi\circ\phi^{-1}$ or $\phi^{-1}\circ\phi$. A finite system of isometries $\mathcal{K}$ has \emph{independent generators} if no reduced word in the isometries in $\Phi$ and their inverses represents a partial isometry of $K$ that fixes some nondegenerate arc. 

Let $T$ be a $(G,\mathcal{F})$-tree, let $K\subseteq T$ be a finite subtree, and let $\Sigma$ be the corresponding foliated complex. We recall from Section \ref{sec-pg-tree} that $X$ is the disjoint union of all trees $K_v$. The quotient space $F:=X/G$ is a finite forest: one connected component of $F$ is isometric to $K_{v_0}$, and the other components are isometric to the quotient spaces $K_{v_i}/G_i$. There is a naturally defined finite system of morphisms on $F$: every map has one base contained in $K_{v_0}$, and there are $N$ morphisms with image in $K_{v_0}$, and for each $i\in\{1,\dots,k\}$, there is a morphism from $K_{v_0}$ to $K_{v_i}/G$. The morphisms of the second kind may however fail to be isometries because two segments in a common base of $K_{v_0}$ may have images in $K_{v_i}$ that belong to the same $G_{v_i}$-orbit, in which case they have the same image in the quotient. However, as mentioned above, by subdividing the bases of these morphisms at the points $x_i\in K_{v_0}$ corresponding to the fixed points of the groups $G_i$, we get a finite system of isometries associated to $K$.
%To get a finite system of isometries on the finite forest $X/G$, we subdivide all bands in $\Sigma$ at the special points of their bases: this operation consists in cutting each band $B$ having a base containing a special point along the leaf that contains this special point, and in this way $B$ is replaced by finitely many new bands (finiteness follows from the fact that the tree $K$ is finite). By doing so, the quotient map from $X$ to $X/G$ becomes injective on bases of bands in $\Sigma$. Therefore, we get a finite system $\mathcal{K}$ of isometries on $X/G$, whose bases are the images of the bases of the bands in $\Sigma$ under the quotient map.

\begin{lemma}\label{lemma-genind}
Let $T$ be a $(G,\mathcal{F})$-tree in which no element of $G$ fixes a nondegenerate arc, and let $K\subseteq T$ be a finite subtree. Then the finite system of isometries associated to $K$ has independent generators.
\end{lemma}

\begin{proof}
We denote by $\mathcal{K}$ this finite system of isometries. Assume towards a contradiction that a reduced word in the partial isometries in $\mathcal{K}$ represents a partial isometry fixing a nondegenerate arc of $X/G$. Then one could find a nondegenerate arc $I\subseteq K$, and an element $g\in G$, such that for all $x\in I$, there is a leaf joining $(x,v_0)$ to $(gx,gv_0)$ in $\Sigma$. But this implies that $I=gI$ in $T_K$. By the fourth property in Theorem \ref{pseudogroup-tree}, there is a morphism from $T_K$ to $T$, so $g$ also fixes a nondegenerate arc in $T$, a contradiction. 
\end{proof}

\section{Dimension of the space of very small $(G,\mathcal{F})$-trees}\label{sec-dimension}

Bestvina and Feighn have shown in \cite[Corollary 7.12]{BF94} that the space of very small $F_N$-trees has dimension $3N-4$. Their result was improved by Gaboriau and Levitt in \cite[Theorem V.2]{GL95}, who showed in addition that $VSL(F_N)\smallsetminus CV_N$ has dimension $3N-5$. Following Gaboriau and Levitt's proof, we extend their computation to the general case of $(G,\mathcal{F})$-trees. We recall the notion of the free rank of $(G,\mathcal{F})$ from Section \ref{sec-free-products}.

\begin{theo}\label{dimension}
Let $G$ be a countable group, and let $\mathcal{F}$ be a free factor system of $G$, such that $\text{rk}_K(G,\mathcal{F})\ge 2$. Then $VSL(G,\mathcal{F})$ has topological dimension $3\text{rk}_f(G,\mathcal{F})+2|\mathcal{F}|-4$.
\end{theo}

\begin{theo}\label{dimension-2}
Let $G$ be a countable group, and let $\mathcal{F}$ be a free factor system of $G$, such that $\text{rk}_K(G,\mathcal{F})\ge 2$. Then $VSL(G,\mathcal{F})\smallsetminus \mathbb{P}\mathcal{O}(G,\mathcal{F})$ has topological dimension $3\text{rk}_f(G,\mathcal{F})+2|\mathcal{F}|-5$.
\end{theo}

The proof of Theorem \ref{dimension} will be carried in the next two sections. Additional arguments for getting the dimension of the boundary (Theorem \ref{dimension-2}) will be given in Section~\ref{sec-add}.

\subsection{The index of a small $(G,\mathcal{F})$-tree}

Let $T$ be a small $(G,\mathcal{F})$-tree, and let $x\in T$. The Kurosh decomposition of the stabilizer $\text{Stab}(x)$ reads as $$\text{Stab}(x)=g_{1}G_{i_1}g_{1}^{-1}\ast\dots\ast g_{r}G_{i_r}g_{r}^{-1}\ast F.$$ We claim that the groups $G_{i_1},\dots,G_{i_r}$ are pairwise non conjugate in $G$, which implies in particular that there are only finitely many free factors arising in the Kurosh decomposition of $\text{Stab}(x)$. Indeed, otherwise, we could find $i\in\{1,\dots,k\}$, and $g\in G\smallsetminus\text{Stab}(x)$, such that both $G_i$ and $gG_ig^{-1}$ fix $x$. This would imply that $G_i$ fixes both $x$ and $g^{-1}x$, and therefore $g^{-1}x=x$ because no arc of $T$ is fixed by a peripheral element. Hence $G_i$ and $gG_ig^{-1}$ are conjugate in $\text{Stab}(x)$, a contradiction. 

Notice that the free group $F$ might \emph{a priori} not be finitely generated (though it will actually follow from Corollary \ref{bound-branch} that it is). We define the \emph{index} of $x$ as $$i(x)=2~ \text{rk}_K(\text{Stab}(x))+v_1(x)-2,$$ where $v_1(x)$ denotes the number of $\text{Stab}(x)$-orbits of directions from $x$ in $T$ with trivial stabilizer. A point $x\in T$ is a \emph{branch point} if $T\smallsetminus\{x\}$ has at least $3$ connected components. It is an \emph{inversion point} if $T\smallsetminus\{x\}$ has $2$ connected components, and some element $g\in G$ fixes $x$ and permutes the two directions at $x$. The following proposition is a generalization of \cite[Proposition III.1]{GL95}.

\begin{prop}\label{index-branch}
For all small minimal $(G,\mathcal{F})$-trees $T$ and all $x\in T$, we have $i(x)\ge 0$. If $T$ is very small, then $i(x)>0$ if and only if $x$ is a branch point or an inversion point.
\end{prop}

\begin{proof}
If $\text{rk}_K(\text{Stab}(x))\ge 2$, then we have $i(x)\ge 2$, and in this case $x$ is a branch point. If $\text{Stab}(x)$ is trivial, then $i(x)=v_1(x)-2$, where $v_1(x)$ is the number of connected components of $T\smallsetminus\{x\}$, which is nonnegative because $T$ is minimal, and $i(x)>0$ if and only if $x$ is a branch point. Finally, if $\text{rk}_K(\text{Stab}(x))=1$, then $i(x)=v_1(x)\ge 0$. If $i(x)>0$, then either $x$ is a branch point as in the first case, or $x$ has valence $2$ and is therefore an inversion point. If $i(x)=0$, and $T$ is very small, the stabilizer of any direction from $x$ is isomorphic to $\text{Stab}(x)$. As tripod stabilizers are trivial in $T$, this implies that $x$ is not a branch point. 
\end{proof}

Let $T$ be a small $(G,\mathcal{F})$-tree, and let $x,x'\in T$. If $x$ and $x'$ belong to the same $G$-orbit, then $i(x)=i(x')$. Given a $G$-orbit $\mathcal{O}$ of points in $T$, we can thus define $i(\mathcal{O})$ to be equal to $i(x)$ for any $x\in\mathcal{O}$. We then let $$i(T):=\sum_{\mathcal{O}\in T/G}i(\mathcal{O}).$$

We now extend \cite[Theorem III.2]{GL95} and its corollaries \cite[Corollaries III.3 and III.4]{GL95} to the context of $(G,\mathcal{F})$-trees. The converse of the second statement of Proposition \ref{bound-index} will be proved in Proposition \ref{index-non-geom} below.

\begin{prop}\label{bound-index}
For all small $(G,\mathcal{F})$-trees $T$, we have $i(T)\le 2\text{rk}_K(G,\mathcal{F})-2$. If $T$ is geometric, then $i(T)=2\text{rk}_K(G,\mathcal{F})-2$.
\end{prop}

Using Proposition \ref{index-branch}, we get the following result as a corollary of Proposition \ref{bound-index}.

\begin{cor}\label{bound-branch}
Any very small $(G,\mathcal{F})$-tree has at most $2 \text{~rk}_K(G,\mathcal{F})-2$ orbits of branch or inversion points, and the Kurosh rank of the stabilizer of any $x\in T$ is at most equal to $\text{rk}_K(G,\mathcal{F})$.
\qed
\end{cor}

\begin{proof}[Proof of Proposition \ref{bound-index}]
First assume that $T$ is geometric. For all $i\in\{1,\dots,k\}$, we let $x_i\in T$ be the fixed point of $G_i$. Let $K\subseteq T$ be a finite subtree containing all points $x_i$, and such that $T=T_{K}$. Let $\mathcal{K}=(F,\Phi)$ be the associated system of morphisms on the compact forest $F:=X/G$ (see Section \ref{sec-systems}). We fix a $G$-orbit of points $\mathcal{O}\subset T$. We will associate to $\mathcal{O}$ two graphs $\mathcal{S}$ and $\mathcal{S}'$ in the following way.  

Vertices of $\mathcal{S}$ are the points in $\mathcal{O}\cap F$ (with a slight abuse of notations, as we are actually considering the image in $F$ of the set of points in $X$ that correspond to points in $\mathcal{O}$). There is an edge $e$ from $z$ to $z'$ whenever there is a morphism in $\Phi$ that sends either $z$ to $z'$ or $z'$ to $z$. We note that Theorem \ref{pseudogroup-tree} implies that $\mathcal{S}$ is connected. The graph $\mathcal{S}$ is actually equal to the quotient space under the $G$-action of the leaf in the foliated $2$-complex $\Sigma$ passing through any point of $\mathcal{O}\cap F$. 

The copies of the points $x_i$ in the trees $K_{v_i}/G_i\subseteq F$ are called the \emph{special vertices} of $F$. We let $n(\mathcal{S})$ be the number of special vertices in $\mathcal{S}$. We define the \emph{Kurosh rank} of $\mathcal{S}$ to be $$\text{rk}_K(\mathcal{S})=1+|E(\mathcal{S})|-|V(\mathcal{S})|+n(\mathcal{S}).$$ We similarly define the Kurosh rank of any subgraph $\mathcal{G}\subseteq\mathcal{S}$.

We will now assign labels in $G$ to the edges in $\mathcal{S}$. When a morphism in $\Phi$ has both its bases in $K_{v_0}$, the corrsponding edge between $z$ and $z'$ comes equipped with a natural (oriented) label, given by the element in the base $B$ of $F_N$ corresponding to the partial isometry. If $z'$ is a nonspecial point in $K_{v_i}/G_i$, its preimage by the corresponding morphism in $\Phi$ consists of finitely many points $z_1,\dots,z_m$ which all differ by elements in $G_i$. We arbitrarily choose the label of one of the corresponding edges (say the one from $z_1$ to $z'$) to be equal to $1$, and if $g_j\in G_i$ satisfies $z_j=g_jz_1$, we label the edge from $z'$ to $z_j$ by $g_j$. We finally assign label $1$ to  every edge that contains the special vertex.  

Fix $x\in\mathcal{O}\cap F$. 

\begin{lemma}\label{lemma-index-1}
We have $\text{rk}_K(\mathcal{S})=\text{rk}_K(\text{Stab}(x))$.
\end{lemma}

\begin{proof}
By adding a vertex with vertex group $G_i$ at each special vertex contained in $\mathcal{S}$, the graph $\mathcal{S}$ has a natural structure of graph of groups, and there is a natural morphism $\rho:\pi_1(\mathcal{S},x)\to G$, sending any edge of a path in $\pi_1(\mathcal{S},x)$ to  the product of the labels it crosses. The morphism $\rho$ is injective and takes its values in $\text{Stab}(x)$, and surjectivity follows from the second assertion of Theorem \ref{pseudogroup-tree}. By definition, the Kurosh rank of $\mathcal{S}$ is also equal to the Kurosh rank of its fundamental group (as a graph of groups), so $\text{rk}_K(\mathcal{S})=\text{rk}_K(\text{Stab}(x))$. 
\end{proof}

We now define a graph $\mathcal{S}'$ by considering orbits of directions instead of orbits of points. Vertices of $\mathcal{S}'$ are the directions in $F$ from points in $\mathcal{O}\cap F$, and two directions $d$ and $d'$ are joined by an edge whenever there exists a morphism in $\Phi$ that either sends $d$ to $d'$ or $d'$ to $d$. In this way, every vertex of $\mathcal{S}$ is replaced by $v_{F}(x)$ vertices in $\mathcal{S}'$, where $v_F(x)$ denotes the valence of $x$ in $F$. Every edge $e$ in $\mathcal{S}$ joining two points $z$ and $z'$, and corresponding to a morphism $\phi$ is replaced by $v_{\phi}(e)$ edges in $\mathcal{S}'$, where $v_{\phi}(e)$ is the valence of $z$ in the domain of $\phi$. There is a natural map $\pi:\mathcal{S}'\to\mathcal{S}$ that sends vertices to vertices and edges to edges. 

\begin{lemma}\label{lemma-index-2}
The set of components of $\mathcal{S}'$ is in one-to-one correspondence with the set of $\text{Stab}(x)$-orbits of directions at $x$ in $T$. For all directions $d\in V(\mathcal{S}')$, the fundamental group of the component of $\mathcal{S}'$ that contains $d$ is isomorphic to the stabilizer of $d$, hence to $\{1\}$ or $\mathbb{Z}$. 
\end{lemma}

\begin{proof}
Let $\mathcal{S}'_1$ be a component of $\mathcal{S}'$. Then $\mathcal{S}'_1$ contains a vertex $d_0\in V(\mathcal{S}'_1)$ corresponding to a direction based at a point $y\in \mathcal{O}\cap K_{v_0}$. Applying any $g\in G$ taking $y$ to $x$, we get a direction $d:=gd_0$ from $x$ in $T$. The $\text{Stab}(x)$-orbit of $d$ only depends on the component $\mathcal{S}'_1$ (and not on the choices of $d_0$ and $g$). This defines a map $\Psi$ from the set of connected components of $\mathcal{S}'$ to the set of $\text{Stab}(x)$-orbits of directions at $x$ in $T$. 

We now prove injectivity of the map $\Psi$. Let $d_0$ and $d'_0$ be two directions in $K=K_{v_0}$ having the same $\Psi$-image, then there exists $g\in G$ mapping $d_0$ to $d'_0$. Theorem \ref{pseudogroup-tree}, applied to two nearby points defining the direction $d_0$, implies that $d_0$ and $d'_0$ belong to the same component of $\mathcal{S}'$, showing injectivity of $\Psi$.

We now show surjectivity of $\Psi$. Let $d$ be a $\text{Stab}(x)$-orbit of directions at $x$ in $T$. There exists a segment $[x,x_1]$, such that $d$ is contained in $[x,x_1]$. Then $[x,x_1]$ is contained in some translate $wK$ with $w\in G$, and $w^{-1}d\subseteq K$. This shows that $d$ belongs to the image of $\Psi$.

The proof of the second statement of the lemma is similar to the proof of Lemma~\ref{lemma-index-1}. For all $d\in V(\mathcal{S}'_1)$, there is an injective morphism $\rho':\pi_1(\mathcal{S}'_1)\to\text{Stab}(d)$, where this time $\mathcal{S}'_1$ is an actual graph and not a graph of groups because no element in $G_i$ fixes a nondegenerate arc in $T_K$. Surjectivity of $\rho'$ follows from Theorem \ref{pseudogroup-tree}, applied to two nearby points defining $d$. 
\end{proof}

We say that a $(G,\mathcal{F})$-tree $T$ has \emph{finitely many orbits of directions} if there are finitely many orbits of directions based at branch or inversion points in $T$. We define the \emph{singular set} $\text{Sing}\subseteq F$ as the subset of $F$ made of all branch points in $F$, all endpoints of the bases of the morphisms in $\Phi$, and all points $x_i$.

\begin{cor}\label{directions}
For all $x\in K$, there are only finitely many $\text{Stab}(x)$-orbits of directions at $x$ in $T$. In addition, there are finitely many orbits of branch or inversion points in $T$ (and hence finitely many orbits of directions in $T$).
\end{cor}

\begin{proof}
Let $d$ be a direction at $x$ in $T$. Then the $\text{Stab}(x)$-orbit of $d$ contains a direction in $K$ based at a point $y\in\mathcal{O}\cap K$. Since $\mathcal{S}$ is connected, there exists $g\in G$ such that $(x,v_0)$ and $(gy,gv_0)$ are joined by a leaf of $\Sigma$. By dragging $d$ along this leaf, we get that either $d$ is a direction at $x$ that is contained in $K$, or $d$ is in the $G$-orbit of a direction in $K$ at a point in the singular set $\text{Sing}$. The claim follows because $\text{Sing}$ is a finite set. Using the fact that the orbit of any point of $T$ meets $K$, the above argument also shows that the orbit of any branch or inversion point in $T$ meets $\text{Sing}$.
\end{proof}

Let $\mathcal{G}$ be a finite connected subgraph of $\mathcal{S}$ containing all vertices in $\mathcal{O}\cap \text{Sing}$ (where $\text{Sing}$ denotes the singular set) and all edges $e\in E(\mathcal{S})$ with $v_{\phi}(e)\neq 2$ (where $\phi$ is the morphism corresponding to $e$). Let $\mathcal{G}'\subseteq\mathcal{S}'$ be the $\pi$-preimage of $\mathcal{G}$ in $\mathcal{S}'$. Lemma \ref{lemma-index-2}, together with the fact that there are only finitely many $\text{Stab}(x)$-orbits of directions at any point $x\in T$ (Corollary \ref{directions}), shows that up to enlarging $\mathcal{G}$ if necessary, we may assume that $\pi_1(\mathcal{G}')$ generates the fundamental group of every component of $\mathcal{S}'$. Denote by $\mathcal{G}'_j$ the components of $\mathcal{G}'$. As the fundamental group of any finite connected graph $X$ satisfies $$1-\text{rk}(\pi_1(X))=|V(X)|-|E(X)|,$$ we have $$\sum_j (1-\text{rk}(\pi_1(\mathcal{G}'_j)))=\sum_{x\in V(\mathcal{G})}v_{F}(x)-\sum_{e\in E(\mathcal{G})}v_{\phi}(e).$$ Moreover, we have $$2\text{rk}_K(\mathcal{G})-2=-2|V(\mathcal{G})|+2| E(\mathcal{G})|+2n(\mathcal{G}).$$ Summing the above two equalities, we get $$2\text{rk}_K(\mathcal{G})-2+\sum_j (1-\text{rk}(\pi_1(\mathcal{G}'_j)))=\sum_{x\in V(\mathcal{G})}(v_F(x)-2)+\sum_{e\in E(\mathcal{G})}(2-v_{\phi}(e))+2n(\mathcal{G}).$$ We claim that $\text{rk}_K(\mathcal{G})$ is bounded independently of the choice of the finite graph $\mathcal{G}$, which implies that $\text{rk}_K(\mathcal{S})$ is finite. Indeed, Lemma \ref{lemma-index-2} implies that $1-\text{rk}(\pi_1(\mathcal{G}'_j))$ cannot be negative. In addition, the right-hand side of the equality does not depend on $\mathcal{G}$, because $v_F(x)=2$ as soon as $x\notin \text{Sing}$, and $v_{\phi}(e)=2$ for all edges of $\mathcal{S}$ that do not belong to $\mathcal{G}$. Up to enlarging $\mathcal{G}$ if necessary, we can thus assume that $\pi_1(\mathcal{G})=\pi_1(\mathcal{S})$ (as graphs of groups). This implies that $\mathcal{G}$ contains any embedded path in $\mathcal{S}$ with endpoints in $\mathcal{G}$, and therefore each component of $\mathcal{S}'$ contains only one component of $\mathcal{G}'$. Lemmas \ref{lemma-index-1} and \ref{lemma-index-2} then imply that the left-hand side of the above equality is equal to $i(\mathcal{O})$, so 
\begin{equation}\label{eq-index}
i(\mathcal{O})=\sum_{x\in V(\mathcal{S})}(v_F(x)-2)+\sum_{e\in E(\mathcal{S})}(2-v_{\phi}(e))+2|\mathcal{F}_{\mathcal{O}}|,
\end{equation}
\noindent where $|\mathcal{F}_{\mathcal{O}}|$ is the common value of $|\mathcal{F}_{|\text{Stab}(x)}|$ for all $x\in\mathcal{O}$ (we recall that $\mathcal{F}_{|\text{Stab}(x)}$ denotes the set of conjugacy classes of peripheral subgroups in the Kurosh decomposition of $\text{Stab}(x)$). We will now sum up the above equality over all orbits of points in $F$ to get an expression of the index of $T$. 
Remark that given a finite tree $K$, we have $$\sum_{x\in K}(v_K(x)-2)=-2,$$ where $v_K(x)$ denotes the valence of $x$ in $K$. Therefore, as $F$ has $1+|\mathcal{F}|$ connected components, we obtain that $$\sum_{\mathcal{O}\in T/G}\sum_{x\in V(\mathcal{S})}(v_F(x)-2)=-2-2|\mathcal{F}|.$$ We also have $$\sum_{\mathcal{O}\in T/G}2|\mathcal{F}_{\mathcal{O}}|=2|\mathcal{F}|.$$ Finally, since the number of morphisms in $\Phi$ is equal to $\text{rk}_K(G,\mathcal{F})$, we have $$\sum_{\mathcal{O}\in T/G}\sum_{e\in\mathcal{S}}(2-v_{\phi}(e))=2\text{rk}_K(G,\mathcal{F}).$$ By summing the above three contributions, we obtain $$i(T)=2\text{rk}_K(G,\mathcal{F})-2,$$ and we are done in the geometric case. 
\\
\\
\indent We now turn to the general case, where $T$ need no longer be geometric. Let $(K_n)_{n\in\mathbb{N}}$ be a sequence of finite subtrees of $T$ such that the corresponding geometric $(G,\mathcal{F})$-trees $T_{n}$ strongly converge to $T$ (we recall the definition of strong convergence from the paragraph preceding Theorem \ref{approx-by-geom}). Let $x\in T$ be a branch or inversion point, and let $s\le i(x)$ be an integer. As $T$ is a $(G,\mathcal{F})$-tree, the Kurosh decomposition of $\text{Stab}(x)$ reads as $$\text{Stab}(x)=g_{1}G_{i_1}g_{1}^{-1}\ast\dots\ast g_{r}G_{i_r}g_{r}^{-1}\ast F,$$ where $G_{i_1},\dots,G_{i_r}$ are pairwise non conjugate in $G$, and $F$ is a free group (which might \emph{a priori} not be finitely generated). Let $Y$ be a finite subset of $\text{Stab}(x)$ made of elements from a free basis of $F$ and one nontrivial element in each of the subgroups $g_{j}G_{i_j}g_{j}^{-1}$. Let $d_1,\dots,d_q$ be directions at $x$ in $T$ with trivial stabilizers, in distinct $\text{Stab}(x)$-orbits. We make these choices in such a way that $2|Y|+q-2=s$. Because of strong convergence, it is possible for $n$ large enough to lift $x$ to an element $x_n\in T_n$ in such a way that all elements in $Y$ fix $x_n$, and we can similarly lift all directions $d_i$ to a direction from $x_n$ in $T_n$. We have $v_1(x_n)\ge q$, and the resolution morphism from $T_n$ to $T$ provided by Theorem \ref{pseudogroup-tree} induces an injective morphism from $\text{Stab}(x_n)$ to $\text{Stab}(x)$, whose image contains all elements in $Y$, and all subgroups $g_{j}G_{i_j}g_{j}^{-1}$ because $T_n$ is a $(G,\mathcal{F})$-tree. As the subgroup generated by $Y$ and the collection of subgroups of the form $g_{j}G_{i_j}g_{j}^{-1}$ is a free factor, we get that $\text{rk}_K(\text{Stab}(x))\ge |Y|$. Hence $i(x_n)\ge s$. As this is true for all $s\le i(x)$, we get that $i(x)\le i(x_n)$. Since lifts to $T_n$ of branch or inversion points in distinct $G$-orbits in $T$ belong to distinct $G$-orbits of $T_n$, it follows from the first part of the argument that $i(T)\le 2\text{~rk}_K(G,\mathcal{F})-2$.
\end{proof}

\subsection{Bounding $\mathbb{Q}$-ranks, and the dimension of $VSL(G,\mathcal{F})$}

We now compute the dimension of $VSL(G,\mathcal{F})$, following the arguments in \cite[Sections IV and V]{GL95}. Let $T$ be a minimal, small $(G,\mathcal{F})$-tree. We denote by $L$ the additive subgroup of $\mathbb{R}$ generated by the values of the translation lengths $||g||_T$, for $g$ varying in $G$. The \emph{$\mathbb{Z}$-rank} $r_{\mathbb{Z}}(T)$ is the rank of the abelian group $L$, i.e. the minimal number of elements in a generating set of $L$ (it is infinite if $L$ is not finitely generated). The \emph{$\mathbb{Q}$-rank} $r_{\mathbb{Q}}(T)$ is defined to be the dimension of the $\mathbb{Q}$-vector space $L\otimes_{\mathbb{Z}}\mathbb{Q}$. Notice that we always have $r_{\mathbb{Q}}(T)\le r_{\mathbb{Z}}(T)$. Let $Y$ be the set of points in $T$ which are either branch points or inversion points. We define $\Lambda$ as the subgroup of $\mathbb{R}$ generated by distances between points in $Y$. We have $2\Lambda\subseteq L\subseteq \Lambda$, see \cite[Section IV]{GL95}. The following two propositions were stated by Gaboriau and Levitt in the case of nonabelian actions of finitely generated groups on $\mathbb{R}$-trees without inversions. Their proofs adapt to our framework.

\begin{prop}\label{rank-1} (Gaboriau--Levitt \cite[Proposition IV.1]{GL95})
Let $T$ be a small $(G,\mathcal{F})$-tree, and let $\{g_1,\dots,g_N\}$ be a free basis of $F_N$. Then the set $\{||g_i||_T\}_{i\in\{1,\dots,N\}}$ generates $L/2\Lambda$.
\end{prop}

\begin{prop} (Gaboriau--Levitt \cite[Proposition IV.1]{GL95})\label{rank-2}
Let $T$ be a small $(G,\mathcal{F})$-tree, and let $\{p_j\}_{j\in J}$ be a set of representatives of the $G$-orbits of branch and inversion points in $T$. Then for all $j_0\in J$, the set $\{d_T(p_{j_0},p_j)\}_{j\in J\smallsetminus\{j_0\}}$ generates $\Lambda/L$.
\end{prop}

We refer the reader to \cite[Proposition IV.1]{GL95} for a proof of the above two facts. We mention that these proofs are based on the following lemma, which follows from standard theory of group actions on $\mathbb{R}$-trees.

\begin{lemma} (Gaboriau--Levitt \cite[Proposition IV.1]{GL95})\label{lemma-trees}
Let $T$ be a small $(G,\mathcal{F})$-tree.
\begin{itemize}
\item For all branch or inversion points $p,q,r\in T$, we have $d(p,r)=d(p,q)+d(q,r) \text{~mod~} 2\Lambda$.
\item For all branch or inversion points $p\in T$ and all $g\in G$, we have $d(p,gp)=||g||_T \text{~mod~} 2\Lambda$.
\item For all $g,h\in G$, we have $||gh||_T=||g||_T+||h||_T \text{~mod~} 2\Lambda$.
\end{itemize}
\end{lemma}

\begin{prop}\label{rank-finite} 
Let $T\in VSL(G,\mathcal{F})$ be a geometric tree, and let $b$ be the number of orbits of branch or inversion points in $T$. Then $r_{\mathbb{Z}}(T)\le \text{rk}_f(G,\mathcal{F})+b-1$.
\end{prop}

\begin{proof}[Proof of Proposition \ref{rank-finite}]
It follows from the proof of Theorem \ref{pseudogroup-tree} that $\Lambda$ is generated by distances between points in the finite singular set $\text{Sing}$. So $\Lambda$ is finitely generated, and therefore $L$ is finitely generated (recall that $2\Lambda\subseteq L\subseteq\Lambda$). Hence $\Lambda/2\Lambda$ is isomorphic to $(\mathbb{Z}/2\mathbb{Z})^{r_{\mathbb{Z}}(T)}$, and the upper bound on $r_{\mathbb{Z}}(T)$ follows from Propositions \ref{rank-1} and \ref{rank-2}.
\end{proof}

We also recall the following result from \cite[Proposition IV.2]{GL95}. 

\begin{prop} (Gaboriau--Levitt \cite[Proposition IV.2]{GL95}) \label{rank-non-geom}
Let $T\in VSL(G,\mathcal{F})$ be a nongeometric tree obtained as the strong limit of a continuous system $T_{\mathcal{K}(t)}$ of geometric trees. Then $$r_{\mathbb{Q}}(T)\le\liminf_{t\to +\infty}r_{\mathbb{Z}}(T_{\mathcal{K}(t)}),$$ and $$r_{\mathbb{Q}}(T)<\limsup_{t\to +\infty}r_{\mathbb{Z}}(T_{\mathcal{K}(t)}).$$
\end{prop}

\begin{prop} \label{rank-bound}
For all $T\in VSL(G,\mathcal{F})$, we have $r_{\mathbb{Q}}(T)\le 3\text{rk}_f(G,\mathcal{F})+2|\mathcal{F}|-3$.
\end{prop}

\begin{proof}
When $T$ is geometric, Proposition \ref{rank-finite} implies that $r_{\mathbb{Q}}(T)\le r_{\mathbb{Z}}(T)\le \text{rk}_f(G,\mathcal{F})+b-1$. Corollary \ref{bound-branch} shows that $b\le 2\text{rk}_f(G,\mathcal{F})+2|\mathcal{F}|-2$, and the claim follows. When $T$ is nongeometric, it is a strong limit of a system of geometric trees. Indeed, the construction of the approximation in Theorem \ref{approx-by-geom} can be done in a continuous way by choosing a continuous increasing family of finite trees $K(t)\subseteq T$, though some of the trees $T_{K(t)}$ may fail to be minimal in general. The claim then follows from Proposition~\ref{rank-non-geom}.
\end{proof}

\begin{prop} (Gaboriau--Levitt \cite[Proposition V.1]{GL95}) \label{rank-ccl}
Let $G$ be a countable group, let $\mathcal{F}$ be a free factor system of $G$, and let $k\ge 1$ be an integer. The space of projectivized length functions of $(G,\mathcal{F})$-trees with $\mathbb{Q}$-rank smaller than or equal to $k$ has dimension smaller than or equal to $k-1$. 
\end{prop}

\begin{proof}[Proof of Theorem \ref{dimension}]
Theorem \ref{dimension} follows from Propositions \ref{rank-bound} and \ref{rank-ccl}, since outer space $\mathbb{P}\mathcal{O}(G,\mathcal{F})$ contains $(3\text{rk}_f(G,\mathcal{F})+2|\mathcal{F}|-4)$-simplices, obtained for instance by varying the edge lengths of a graph of groups that has the shape displayed on Figure~\ref{fig-interior}. 
\end{proof}

\begin{figure}
\begin{center}
\input{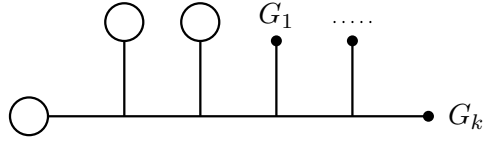}
\caption{A $(3\text{rk}_f(G,\mathcal{F})+2|\mathcal{F}|-4)$-simplex in $\mathbb{P}\mathcal{O}(G,\mathcal{F})$.}
\label{fig-interior}
\end{center}
\end{figure}

\subsection{Very small graphs of actions}\label{sec-Levitt}

In this section, we mention a decomposition result which was proved by Levitt for actions of finitely generated groups on $\mathbb{R}$-trees having finitely many orbits of branch points \cite[Theorem 1]{Lev94}. The proof uses the fact that every such action on a tree $T$ is \emph{finitely supported}, i.e. there exists a finite tree $K\subset T$ such that every arc $I\subset T$ is covered by finitely many translates of $K$. The fact that minimal $(G,\mathcal{F})$-actions on $\mathbb{R}$-trees are finitely supported was noticed by Guirardel in \cite[Lemma 1.14]{Gui08}. Using finiteness of the number of orbits of branch and inversion points in a very small $(G,\mathcal{F})$-tree, Levitt's theorem adapts to our more general framework.

\begin{theo} (Levitt \cite[Theorem 1]{Lev94})\label{goa}
Let $G$ be a countable group, and let $\mathcal{F}$ be a free factor system of $G$. Then every tree $T\in VSL(G,\mathcal{F})$ splits uniquely as a graph of actions, all of whose vertex trees have dense orbits, such that the Bass--Serre tree of the underlying graph of groups is very small, and all its edges have positive length.
\end{theo}

\subsection{Additional arguments for computing the dimension of $VSL(G,\mathcal{F})\smallsetminus \mathbb{P}\mathcal{O}(G,\mathcal{F})$}\label{sec-add}

We start by recalling the following well-known fact. We recall that a $(G,\mathcal{F})$-tree $T$ has \emph{finitely many orbits of directions} if there are finitely many orbits of directions at branch or inversion points in $T$.

\begin{prop}\label{dense-arcs}
Let $T$ be a $(G,\mathcal{F})$-tree with dense orbits. If $T$ is small, and has finitely many orbits of directions (in particular, if $T$ is very small, or small and geometric), then all stabilizers of nondegenerate arcs in $T$ are trivial.
\end{prop}

\begin{proof}
Let $e\subseteq T$ be a nondegenerate arc in $T$, and assume there exists a nontrivial element $g\in G$ such that $ge=e$. We can find two distinct directions $d,d'$ in $e$ based at branch or inversion points of $T$ (oriented in the same way), and an element $h\in G$ so that $d'=hd$. Notice in particular that $h$ is hyperbolic in $T$, so $h$ is nonperipheral, and $\langle g,h\rangle$ is not cyclic. The points at which these directions are based can be chosen to be both arbitrarily close to the midpoint of $e$, and in this case $g$ and $hgh^{-1}$ fix a common nondegenerate subarc of $e$. As $T$ is small, this implies that $g$ and $hgh^{-1}$ commute. Hence $h$ preserves the axis of $g$ in any Grushko $(G,\mathcal{F})$-tree, which implies that $g$ and $h$ generate a cyclic subgroup of $G$, a contradiction. 
\end{proof}

The following proposition extends \cite[Theorem III.2]{GL95}.

\begin{prop}\label{index-non-geom}
Let $T$ be a small $(G,\mathcal{F})$-tree. If $T$ is nongeometric, then $i(T)<2\text{rk}_K(G,\mathcal{F})-2$.
\end{prop}

\begin{proof}
We know from Proposition \ref{bound-index} that $i(T)\le 2\text{rk}_K(G,\mathcal{F})-2$. Assume towards a contradiction that $i(T)=2\text{rk}_K(G,\mathcal{F})-2$.

Let $Y\subset T$ be a finite set that contains one point from each $G$-orbit with positive index, and let $x\in Y$. The Kurosh decomposition of $\text{Stab}(x)$ reads as $$\text{Stab}(x)=H_{i_1}\ast\dots\ast H_{i_j}\ast F,$$ where $F$ is a finitely generated free group, and $H_{i_l}$ is $G$-conjugate to $G_{i_l}$ for all $l\in\{1,\dots,j\}$. Let $q$ be the rank of $F$. Let $(T_n)_{n\in\mathbb{N}}$ be an approximation of $T$ constructed as in the proof of Theorem \ref{approx-by-geom}. We can assume that $K_n$ has been chosen so that the extremities of $K_n$ are branch points or inversion points in $T$ (this can be achieved by choosing for $x_0$ a branch point or inversion point of $T$, with the notations from the proof of Theorem \ref{approx-by-geom}). As in the proof of Proposition \ref{bound-index} in the nongeometric case, we choose directions $d_1,\dots,d_r$, so that $2j+2q+r-2=i(x)$, and $n\in\mathbb{N}$ so that we can associate a point $x_n\in T_n$ to each $x\in Y$.

As $i(T_n)=i(T)$, the orbit of every branch or inversion point of $T_n$ with positive index contains some $x_n$. Furthermore, every direction from $x_n$ with trivial stabilizer belongs to the $\text{Stab}(x_n)$-orbit of the lift $d'_{\beta}$ of one of the directions $d_{\beta}$ to $T_n$. 

The morphism $j_n:T_n\to T$ is not an isometry, otherwise $T$ would be geometric. Hence there exist $y\in T_n$, and two adjacent arcs $e_1$ and $e_2$ at $y$ whose $j_n$-images have a common initial segment. If $y$ is a branch point or an inversion point with positive index, it follows from the above paragraph that both $e_1$ and $e_2$ have nontrivial stabilizer (otherwise we would have $i(T)<i(T_n)$). As $T$ is small, the stabilizers of $e_1$ and $e_2$ generate a cyclic subgroup of $G$, so there exists $g\in G$ that fixes both $e_1$ and $e_2$ in $T_n$. This contradicts injectivity of $j_n$ in restriction to the fixed point set of $g$ (Proposition~\ref{stab-isom}). If $y$ is a branch or inversion point with index $0$, then $y$ has cyclic stabilizer, and there exists $g\in G$ that stabilizes all adjacent edges. Again, this contradicts injectivity of $j_n$ in restriction to the fixed point set of $g$. If $y$ is neither a branch point nor an inversion point, then Theorem~\ref{pseudogroup-tree} implies that $e_1$ and $e_2$ are contained in the interior of a common $G$-translate of $K_n$, because extremal points of $K_n$ have been chosen to be branch or inversion points in $T$. This again leads to a contradiction, since the restriction of $j_n$ to this translate of $K_n$ is an isometry.
\end{proof}

The following proposition is an extension of \cite[Theorem IV.1]{GL95}.

\begin{prop}\label{rank-bnd}
For all very small $(G,\mathcal{F})$-trees $T\notin\mathcal{O}(G,\mathcal{F})$, we have $r_{\mathbb{Q}}(T)< 3\text{rk}_f(G,\mathcal{F})+2|\mathcal{F}|-3$.
\end{prop}

\begin{proof}
When $T$ is nongeometric, the claim follows from Propositions \ref{rank-non-geom} and \ref{rank-bound}. We will assume that $T$ is geometric and show that $r_{\mathbb{Z}}(T)<3\text{rk}_f(G,\mathcal{F})+2|\mathcal{F}|-3$. We have $r_{\mathbb{Z}}(T)<+\infty$ (Proposition \ref{rank-finite}), and $\Lambda/2\Lambda$ is isomorphic to $(\mathbb{Z}/2\mathbb{Z})^{r_{\mathbb{Z}}(T)}$.  

If the number of distinct orbits of branch or inversion points in $T$ is strictly smaller than $2\text{rk}_K(G,\mathcal{F})-2$, then $r_{\mathbb{Z}}(T)<3\text{rk}_f(G,\mathcal{F})+2|\mathcal{F}|-3$ by Proposition \ref{rank-finite}, and we are done. Otherwise, let $p_1,\dots,p_{2\text{rk}_K(G,\mathcal{F})-2}$ be a set of representatives in $K$ of the orbits of branch or inversion points in $T$. Proposition \ref{bound-index} implies that for all $j\in\{1,\dots,2\text{rk}_K(G,\mathcal{F})-2\}$, we have $i(p_j)\le 1$, and hence $i(p_j)=1$ by Proposition \ref{index-branch}.

If $T$ is a simplicial tree, then $\Lambda$ is generated by the lengths of the edges of the quotient graph of groups. In particular, Proposition \ref{rank-bound} implies that the maximal number of edges of a simplicial tree in $VSL(G,\mathcal{F})$ is $3\text{rk}_f(G,\mathcal{F})+2|\mathcal{F}|-3$. All vertices of $T$ have index $1$. Therefore, if $x\in T$ is a vertex, we either have $\text{Stab}(x)=\{e\}$ and $v_1(x)=3$, or $\text{rk}_K(\text{Stab}(x))=1$ and $v_1(x)=1$. Using the fact that $T$ is very small, we get that every vertex $v$ of $T$ satisfies one of the following possibilities, displayed on Figure \ref{fig-possibilities}: either $v$
\begin{enumerate}
\item has valence $3$, and trivial stabilizer, or
\item projects in the quotient graph of groups to a valence $1$ vertex whose stabilizer is peripheral, or
\item projects in the quotient graph of groups to a valence $1$ vertex whose stabilizer is isomorphic to $\mathbb{Z}$, and not peripheral, or
\item projects in the quotient graph of groups to a valence $2$ vertex with stabilizer isomorphic to $\mathbb{Z}$ and not peripheral, adjacent to both an edge with trivial stabilizer and an edge with $\mathbb{Z}$ stabilizer, or
\item projects in the quotient graph of groups to a valence $3$ vertex with stabilizer isomorphic to $\mathbb{Z}$ and not peripheral, adjacent to one edge with trivial stabilizer, and two edges with $\mathbb{Z}$ stabilizers.
\end{enumerate}

\begin{figure}
\begin{center}
\input{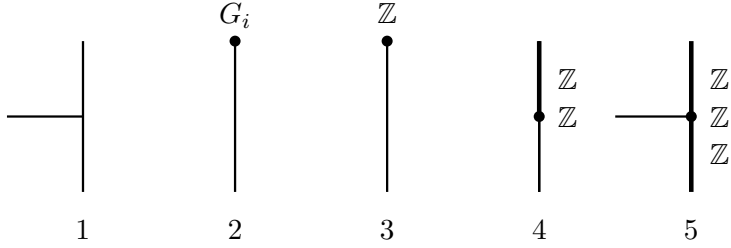}
\caption{Vertices of index $1$ in a very small simplicial $(G,\mathcal{F})$-tree.}
\label{fig-possibilities}
\end{center}
\end{figure}

\noindent As $T\notin \mathcal{O}(G,\mathcal{F})$, some vertex in $T$ satisfies one of the last three possibilities. If some vertex in $T$ satisfies the third possibility, then one can replace the point with $\mathbb{Z}$-stabilizer in the quotient graph of groups by a loop-edge. This operation yields a new minimal, very small simplicial tree $T'$ for which $r_{\mathbb{Z}}(T')>r_{\mathbb{Z}}(T)$, so $r_{\mathbb{Z}}(T)<3\text{rk}_f(G,\mathcal{F})+2|\mathcal{F}|-3$. Otherwise, the graph of groups $T/G$ contains a concatenation of edges that all have the same $\mathbb{Z}$ stabilizer, whose two extremal vertices have valence $2$, and are adjacent to an edge with trivial stabilizer, and whose interior vertices have valence $3$, and are adjacent to a single edge with trivial stabilizer, see Figure \ref{fig-sun}. Figure \ref{fig-sun} illustrates how to construct a tree $T'$ with strictly more orbits of edges than $T$, so that $r_{\mathbb{Z}}(T')>r_{\mathbb{Z}}(T)$. Again, we have $r_{\mathbb{Z}}(T)<3\text{rk}_f(G,\mathcal{F})+2|\mathcal{F}|-3$.

\begin{figure}
\begin{center}
\input{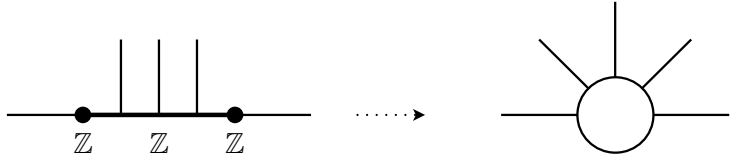}
\caption{Simplicial trees in $VSL(G,\mathcal{F})\smallsetminus \mathbb{P}\mathcal{O}(G,\mathcal{F})$ do not have maximal $\mathbb{Z}$-rank.}
\label{fig-sun}
\end{center}
\end{figure}

Assume now that $T$ has dense orbits. Notice that $\Lambda/2\Lambda=\Lambda/L+L/2\Lambda$, so by Propositions \ref{rank-1} and \ref{rank-2}, it suffices to prove that the rank of $\Lambda/L$ is strictly less than $b-1$, where $b$ denotes the number of orbits of branch or inversion points in $T$. Let $K\subseteq T$ be a finite subtree such that $T=T_K$, chosen in such a way that every terminal vertex of $K$ is either a branch point or an inversion point in $T$. The $(G,\mathcal{F})$-tree $T$ has trivial arc stabilizers by Proposition \ref{dense-arcs}, so the associated system of isometries $\mathcal{K}=(F,\Phi)$  has independent generators by Lemma \ref{lemma-genind} (notice here that we are considering the system of isometries and not the system of morphisms). Using \cite[Proposition 6.1]{GLP94}, we get that 
\begin{equation}\label{eq-independent}
|F|=\sum_{\phi\in \Phi} |A_{\phi}|,
\end{equation}
\noindent where $|F|$ (resp. $|{A_{\phi}}|$) denotes the total length of $F$ (resp. of ${A_{\phi}}$). We have $$|F|=\sum_{e=[q,r]\in E(K)}d_T(q,r),$$ where the sum is taken over all edges in $F$, after subdividing $F$ at all points in the singular set $\text{Sing}$ (notice that in view of our choice of $K$, all points in $\text{Sing}$ correspond to branch points or inversion points in $T$). We denote by $V(F)$ the set of vertices in $F$. Our hypothesis on the extremal vertices of $K$, together with Lemma \ref{lemma-trees}, implies that for all points $q,r\in V(F)$, the length of $[q,r]$ is equal modulo $L$ to the sum $d_T(p_1,p_i)+d_T(p_1,p_j)$, where $p_i$ (resp. $p_j$) belongs to the $G$-orbit of $q$ (resp. $r$). Denoting by $\mathcal{O}(p_i)$ the orbit of $p_i$ for all $i\in\{1,\dots,2\text{rk}_K(G,\mathcal{F})-2\}$, we have $$|F|=\sum_{i=1}^{2\text{rk}_K(G,\mathcal{F})-2}d_T(p_1,p_i)\times\left(\sum_{x\in V(F)\cap\mathcal{O}(p_i)}v_F(x)\right) \text{~mod~}L,$$ and similarly, for all $\phi\in\Phi$, we have $$|A_{\phi}|=\sum_{i=1}^{2\text{rk}_K(G,\mathcal{F})-2}d_T(p_1,p_i)\times\left(\sum_{x\in V(A_{\phi})\cap \mathcal{O}(p_i)}v_{\phi}(x)\right) \text{~mod~} L,$$ where $V(A_{\phi})$ denotes the set of vertices of $A_{\phi}$, which is contained in $V(F)$, and $v_{\phi}(x)$ denotes the valence of $x$ in $A_{\phi}$. Using the above two equalities, Equation \eqref{eq-independent} gives a linear relation in $\Lambda/L$ between the numbers $d_T(p_1,p_i)$, where the coefficient of $d_T(p_1,p_i)$ is equal to $$\sum_{x\in V(F)\cap\mathcal{O}(p_i)}\left( v_F(x)-\sum_{\phi\in\Phi}v_{\phi}(x)\right).$$ For all $i\in\{1,\dots,2\text{rk}_K(G,\mathcal{F})-2\}$, the index of $p_i$ is equal to $1$. Therefore, Equation~\eqref{eq-index} from the proof of Proposition \ref{bound-index} implies that $$\sum_{x\in V(F)\cap\mathcal{O}(p_i)}\left( v_F(x)-\sum_{\phi\in\Phi}v_{\phi}(x)\right)$$ is odd. Equation \eqref{eq-independent} thus leads to the nontrivial relation $$\sum_{j=2}^{2\text{rk}_K(G,\mathcal{F})-2} d_T(p_1,p_j)=0 \text{~mod~} L$$ between the generators of $\Lambda/L$, so $r_{\mathbb{Z}}(T)<3\text{rk}_f(G,\mathcal{F})+2|\mathcal{F}|-3$. 

In general, let $\mathcal{G}$ be the decomposition of $T$ as a graph of actions provided by Theorem~ \ref{goa}. We assume that $T$ is not simplicial, and let $T_v$ be a nontrivial vertex tree of this decomposition. Then $T_v$ is a very small $(G_v,\mathcal{F}_{G_v})$-tree with dense $G_v$-orbits. Let $T'$ be the very small $(G,\mathcal{F})$-tree obtained from $T$ by collapsing all vertex trees in the $G$-orbit of $T_v$ to points. By definition of the index, we have $$i(T)-i(T')=i(T_v)-(2\text{rk}_K(G_v)-2).$$ As $T$ is geometric, Proposition \ref{bound-index} implies that the left-hand side of the above equality is nonnegative, while the right-hand side is nonpositive. This implies that $i(T_v)=2\text{rk}_K(G_v)-2$. Using Proposition \ref{index-non-geom}, this shows that the tree $T_v$ is geometric. Assume that the number of distinct $G_v$-orbits of branch or inversion points in the minimal subtree of $T_v$ is strictly smaller than $2\text{rk}_K(G_v)-2$. Then one of these orbits has index at least $2$ in $T_v$, and hence in $T$. This implies that the number of distinct $G$-orbits of branch or inversion points in $T$ is strictly smaller than $2\text{rk}_K(G,\mathcal{F})-2$, and we are done by Proposition \ref{rank-finite}. We are thus left with the case where the number of distinct $G_v$-orbits of branch or inversion points in $T_v$ is equal to $2\text{rk}_K(G_v)-2$. As distinct $G$-translates of $T_v$ are disjoint in $T$, all these $G_v$-orbits of branch or inversion points are distinct when viewed as $G$-orbits of points in $T$. We denote by $p_1,\dots,p_{2\text{rk}_K(G_v)-2}$ a set of representatives of the $G_v$-orbits of branch or inversion points of $T_v$. In this case, as $T_v$ has dense orbits, the analysis from the above paragraph provides a nontrivial relation between the generators $d_{T_v}(p_1,p_i)$ of $\Lambda(T_v)/L(T_v)$. The numbers $d_{T_v}(p_1,p_i)$ may also be viewed as part of a generating set of $\Lambda(T)/L(T)$, and we have a nontrivial relation between these generators. Again, this implies that $r_{\mathbb{Z}}(T)<3\text{rk}_K(G,\mathcal{F})+2|\mathcal{F}|-3$. 
\end{proof}

\begin{proof}[Proof of Theorem \ref{dimension-2}]
Theorem \ref{dimension-2} follows from Propositions \ref{rank-ccl} and \ref{rank-bnd}, because $VSL(G,\mathcal{F})\smallsetminus \mathbb{P}\mathcal{O}(G,\mathcal{F})$ contains a $(3\text{rk}_f(G,\mathcal{F})+2|\mathcal{F}|-5)$-simplex made of simplicial $(G,\mathcal{F})$-trees (except in the case where $G=G_1\ast G_2$ and $\mathcal{F}=\{G_1,G_2\}$, for which $\mathbb{P}\mathcal{O}(G,\mathcal{F})$ is reduced to a point and $VSL(G,\mathcal{F})\smallsetminus \mathbb{P}\mathcal{O}(G,\mathcal{F})$ is empty). An example of such a simplex is given by varying edge lengths in a graph of groups obtained from the graph of groups displayed on Figure \ref{fig-interior} by collapsing a loop, or merging two points corresponding to subgroups $G_1$ and $G_2$, and adding an edge with nontrivial cyclic stabilizer generated by a nonperipheral element in $G_1\ast G_2$, for instance.
\end{proof}

\section{Very small actions are in the closure of outer space.}\label{sec-closure}

In the classical case where $G=F_N$ is a finitely generated free group of rank $N$, and $\mathcal{F}=\emptyset$, Cohen and Lustig have shown that a minimal, simplicial $F_N$-tree lies in the closure $\overline{cv_N}$ if and only if it is very small \cite{CL95}. Bestvina and Feighn \cite{BF94} have extended their result to all minimal $F_N$-actions on $\mathbb{R}$-trees. However, it seems that their proof does not handle the case of actions that contain both nontrivial arc stabilizers, and minimal components dual to measured foliations on compact, nonorientable surfaces. Indeed, for such actions, it is not clear how to approximate the foliation by rational ones without creating any one-sided leaf (in which case the action we get is not very small). If the action has trivial arc stabilizers (i.e. if the dual band complex contains no annulus), then the argument in \cite[Lemma 4.1]{BF94} still enables to get an approximation by very small, simplicial $F_N$-trees, by using the narrowing process described in \cite[Section 7]{Gui98}. However, this argument does not seem to handle the case of trees having nontrivial arc stabilizers. We will give a proof of the fact that $\overline{cv_N}$ is the space of very small, minimal, isometric actions of $F_N$ on $\mathbb{R}$-trees that does not rely on train-track arguments for approximating measured foliations on surfaces by rational ones. Our proof also gives an interpretation of Cohen and Lustig's for simplicial trees. We again work in our more general framework of $(G,\mathcal{F})$-trees, and show the following result.

\begin{theo}\label{cv-closure}
Let $G$ be a countable group, and let $\mathcal{F}$ be a free factor system of $G$. The closure $\overline{\mathcal{O}(G,\mathcal{F})}$ (resp. $\overline{\mathbb{P}\mathcal{O}(G,\mathcal{F})}$) is the space of (projective) length functions of very small $(G,\mathcal{F})$-trees.
\end{theo}

In particular, Theorem \ref{cv-closure} states that every minimal, very small $(G,\mathcal{F})$-tree $T$ can be approximated by a sequence of Grushko $(G,\mathcal{F})$-trees. When $T$ has trivial arc stabilizers, we can be a bit more precise about the nature of the approximation we get. 

\begin{de}
Let $T$ be a $(G,\mathcal{F})$-tree. A \emph{Lipschitz approximation} of $T$ is a sequence $(T_n)_{n\in\mathbb{N}}$ of $(G,\mathcal{F})$-trees that converges to $T$, and such that for all $n\in\mathbb{N}$, there exists a $1$-Lipschitz $G$-equivariant map from $T_n$ to $T$.
\end{de}

Lipschitz approximations seem to be useful: they were a crucial ingredient in \cite{Hor14-1} for tackling the question of spectral rigidity of the set of primitive elements of $F_N$ in $\overline{cv_N}$. They also turn out to be a useful ingredient for describing the Gromov boundary of the (hyperbolic) graph of $(G,\mathcal{F})$-cyclic splittings in \cite{Hor14-6}. The \emph{quotient volume} of a very small $(G,\mathcal{F})$-tree $T$ is defined as the sum of the edge lengths of the graph of actions underlying the Levitt decomposition of $T$ (where we recall the definition of the Levitt decomposition from Theorem \ref{goa}). In other words, if $S$ is the simplicial tree obtained after collapsing all subtrees with dense orbits in the decomposition of $T$, then the quotient volume of $S$ is the volume of the quotient graph $S/G$. In particular, if $T$ has dense orbits, then its quotient volume is equal to $0$.

\begin{theo}\label{Lip-approx}
Let $G$ be a countable group, and let $\mathcal{F}$ be a free factor system of $G$. Then every minimal $(G,\mathcal{F})$-tree $T$ with trivial arc stabilizers admits a Lipschitz approximation by (unprojectivized) Grushko $(G,\mathcal{F})$-trees, whose quotient volumes converge to the quotient volume of $T$.
\end{theo}

\subsection{Reduction lemmas}\label{sec-reduction}

To prove Theorem \ref{cv-closure}, we are left showing that every very small minimal $(G,\mathcal{F})$-tree $T$ can be approximated by a sequence of Grushko $(G,\mathcal{F})$-trees. By Theorem \ref{approx-by-geom-2}, we can approximate every minimal, very small $(G,\mathcal{F})$-tree $T$ by a sequence of minimal, very small, geometric $(G,\mathcal{F})$-trees. This approximation is a Lipschitz approximation. If $T$ has dense orbits, then Remark \ref{rk-vol} implies that the quotient volumes of the trees in the approximation converge to the quotient volume of $T$, which is equal to $0$. More generally, in the case where the Levitt decomposition of $T$ has trivial edge stabilizers, then Lemmas \ref{Guirardel-reduction} and \ref{reduction-2} below enable us to find a Lipschitz approximation of $T$ by trees $T_n$ whose quotient volumes converge to the quotient volume of $T$, such that all trees with dense orbits in the Levitt decomposition of the trees $T_n$ are geometric. In this case, the argument from the last paragraph of the proof of Proposition \ref{rank-bnd} shows that the trees $T_n$ themselves are geometric. To complete the proof of Theorems \ref{cv-closure} and \ref{Lip-approx}, we are left understanding how to approximate minimal, very small, geometric $(G,\mathcal{F})$-trees by minimal Grushko $(G,\mathcal{F})$-trees. 

Our proof of Theorems \ref{cv-closure} and \ref{Lip-approx} will make use of the following lemmas, which enable us to approximate very small $(G,\mathcal{F})$-trees that split as graphs of actions, as soon as we are able to approximate the vertex actions. Lemma \ref{reduction-2} is a version of Guirardel's Reduction Lemma in \cite[Section 4]{Gui98}, where we keep track of the fact that the approximations of the trees are Lipschitz approximations. In Lemma \ref{reduction}, we tackle the problem of approximating trees with nontrivial arc stabilizers by Grushko $(G,\mathcal{F})$-trees. Our argument may be seen as an interpretation of Cohen and Lustig's twisting argument for approximating such trees \cite{CL95}. We consider graphs of actions, instead of restricting ourselves to simplicial trees. We first recall Guirardel's Reduction Lemma from \cite[Section 4]{Gui98}. In the statements below, all limits are nonprojective.

\begin{lemma}\label{Guirardel-reduction}(Guirardel \cite[Section 4]{Gui98})
Let $T$ be a very small $(G,\mathcal{F})$-tree that splits as a graph of actions $\mathcal{G}$. Assume that all pointed vertex actions $(T^v,(u_1^v,\dots,u_k^v))$ admit an approximation by a sequence of pointed $(G^v,\mathcal{F}_{|G^v})$-actions $((T^v_n,(u^v_{1,n},\dots,u^v_{k,n})))_{n\in\mathbb{N}}$, in which the approximation points are fixed by the adjacent edge stabilizers. For all $n\in\mathbb{N}$, let $T_n$ be the $(G,\mathcal{F})$-tree obtained by replacing all vertex actions $(T^v,(u_1^v,\dots,u_k^v))$ by their approximation $(T^v_n,(u^v_{1,n},\dots,u^v_{k,n}))$ in $\mathcal{G}$. Then $(T_n)_{n\in\mathbb{N}}$ converges to $T$. 
\end{lemma}

We say that a sequence $((T_n,(u_n^1,\dots,u_n^k)))_{n\in\mathbb{N}}$ of pointed $(G,\mathcal{F})$-trees is a \emph{Lipschitz approximation} of a pointed $(G,\mathcal{F})$-tree $(T,(u^1,\dots,u^k))$ if $((T_n,(u_n^1,\dots,u_n^k)))_{n\in\mathbb{N}}$ converges to $(T,(u^1,\dots,u^k))$, and for all $n\in\mathbb{N}$, there exists a $1$-Lipschitz $G$-equivariant map $f_n:T_n\to T$ such that for all $i\in\{1,\dots,k\}$, we have $f_n(u^i_n)=u^i$. Guirardel's Reduction Lemma can be refined in the following way.
  
\begin{lemma}\label{reduction-2}
Let $T$ be a very small $(G,\mathcal{F})$-tree with trivial arc stabilizers, that splits as a $(G,\mathcal{F})$-graph of actions $\mathcal{G}$. If all pointed vertex trees $(T^v,(u_1^v,\dots,u_k^v))$ of $\mathcal{G}$ admit Lipschitz approximations by pointed Grushko $(G^v,\mathcal{F}_{|G^v})$-trees, in which the approximation points $u^v_{1,n},\dots,u^v_{k,n}$ are fixed by the adjacent edge stabilizers, then $T$ admits a Lipschitz approximation by Grushko $(G,\mathcal{F})$-trees. 
\qed
\end{lemma}

\begin{lemma}\label{reduction}
Let $T$ be a minimal, very small $(G,\mathcal{F})$-tree, that splits as a $(G,\mathcal{F})$-graph of actions over a one-edge $(G,\mathcal{F})$-free splitting (where the vertex actions need not be minimal). If the minimal subtrees of all vertex trees of $\mathcal{G}$ (with respect to the action of their stabilizer $G^v$) admit approximations by minimal Grushko $(G^v,\mathcal{F}_{|G^v})$-trees, then $T$ admits an approximation by minimal Grushko $(G,\mathcal{F})$-trees. 
\end{lemma}

Figures \ref{fig-1-shape}, \ref{fig-approx-21}, and \ref{fig-approx-23} provide examples of trees for which the vertex actions of the splitting are not minimal (but they are minimal in the sense of pointed trees when we keep track of the attaching points). These are the crucial cases of Lemma \ref{reduction}, in which we deal with the problem of approximating trees with nontrivial arc stabilizers. Considering non-minimal vertex actions is crucial to deal with the simplicial case in Theorem \ref{cv-closure} (when there are edges with nontrivial stabilizers), and Lemma \ref{reduction} provides a new interpretation of Cohen and Lustig's argument for dealing with this case. Lemma~\ref{reduction} will also be crucial for dealing with the case of geometric actions of surface type containing nontrivial arc stabilizers. 

\begin{proof}
We will provide a detailed argument in the case where the $(G,\mathcal{F})$-free splitting $S$ is a free product, and explain how to adapt the argument to the case of an HNN extension.
\\ 
\\
\textit{Case 1} : The splitting $S$ is of the form $G=A\ast B$.
\\
\\
The following description of $\mathcal{G}$ is illustrated in Figure \ref{fig-1-shape}. We denote by $L$ the length of the edge of $\mathcal{G}$, which might be equal to $0$. Denote by $T^A$ and $T^B$ the vertex trees of $\mathcal{G}$, and by $u^A\in T^A$ and $u^B\in T^B$ the corresponding attaching points. The trees $T^A$ and $T^B$ may fail to be minimal, we denote by $T^A_{min}$ and $T^B_{min}$ their minimal subtrees. Up to enlarging $L$ if necessary, we can assume that the set $T^A\smallsetminus T^A_{min}$ is either empty (in the case where $T^A$ is minimal), or consists of the orbit of a single point in the closure of $T^{A}_{min}$, or consists of the orbit of a nondegenerate half-open arc with nontrivial stabilizer. 

\begin{figure}
\begin{center}
\input{1-shape.pst}
\caption{The splitting of $T$ as a graph of actions in Case 1 of the proof of Lemma \ref{reduction}.}
\label{fig-1-shape}
\end{center}
\end{figure}

We will explain how to approximate the tree $(T^A,u^A)$ by a sequence of pointed Grushko $(A,\mathcal{F}_{|A})$-trees. By approximating the pointed tree $(T^B,u^B)$ in the same way, our claim then follows from Guirardel's Reduction Lemma (Lemma \ref{Guirardel-reduction}). 

If $T^A$ is minimal, we can approximate $(T^A,u^A)$ by a sequence of pointed Grushko $(A,\mathcal{F}_{|A})$-trees $(T^A_n,u_n^A)$ by assumption (by choosing $u_n^A$ to be an approximation of $u^A$ in the tree $T^A_n$, provided by the definition of the equivariant Gromov--Hausdorff topology). This also remains true in the case where $T^A\smallsetminus T^A_{min}$ consists of the orbit of a single point $u^A$ in the closure of $T^{A}_{min}$. Indeed, in this case, we can first approximate $u^A$ by a sequence of points $(u'_n)_{n\in\mathbb{N}}\in (T^A_{min})^{\mathbb{N}}$, and then choose for each $n\in\mathbb{N}$ an approximation $u_n^A$ of $u'_n$ in an approximation of $T^A_{min}$.

We now assume that $\overline{T^A\smallsetminus T^A_{min}}$ consists of the orbit of a nondegenerate arc $[u^A,v^A]$ with nontrivial stabilizer, whose length we denote by $l^A$. We choose the notations so that $v^A\in T^A_{min}$. We will also first assume that $T^A_{min}$ is not reduced to a point. As $T$ is very small, the stabilizer $\langle c^A\rangle$ of the arc $[u^A,v^A]$ is cyclic, closed under taking roots, and non-peripheral. As tripod stabilizers are trivial in $T$, the point $v^A$ is an endpoint of the subarc of $T^A_{min}$ fixed by $c^A$. If this arc is nondegenerate, then we let $w^A$ be its other endpoint. Otherwise, we let $w^A$ be any point that is not equal to $v^A$. 

Let $(T^A_{min,n})_{n\in\mathbb{N}}$ be an approximation of $T^A_{min}$ by minimal Grushko $(A,\mathcal{F}_{|A})$-trees. Denote by $v^A_n$ (respectively $w^A_n$) an approximation of $v^A$ (resp. $w^A$) in the tree $T^A_{min,n}$, provided by the definition of convergence in the equivariant Gromov--Hausdorff topology. We can assume that for all $n\in\mathbb{N}$, the point $v^A_n$ belongs to the axis of $c^A$ in $T^A_{min,n}$.   

We refer to Figure \ref{fig-1-shape-2} for an illustration of the following construction. For all $n\in\mathbb{N}$, let $(T_n^A,u_n^A)$ be the pointed tree obtained from $(T^A,u^A)$ in the following way. We start by equivariantly unfolding the arc $[u^A,v^A]$ to obtain a tree $\widetilde{T}^A$ that contains an edge $e^0$ of length $l^A$ with trivial stabilizer. We then equivariantly replace the pointed tree $(T^A_{min},v^A)$ in the graph of actions defining $\widetilde{T}^A$ by its approximation $(T^A_{min,n},v^A_n)$, to get a tree $\widetilde{T_n}^A$. Finally, we define the tree $(T_n^A,u_n^A)$ in the following way: the tree $(T_n^A,u_n^A)$ is obtained from $(\widetilde{T_n}^A,u^A)$ by fully folding the edge $e^0$ along the axis of $c^A$ in $T^A_{min,n}$, in a direction that does not contain $w_n^A$. We denote by $f^A_n:(\widetilde{T_n}^A,u^A)\to (T^A_n,u_n^A)$ the folding map. 

\begin{figure}
\begin{center}
\def\JPicScale{.95}
\input{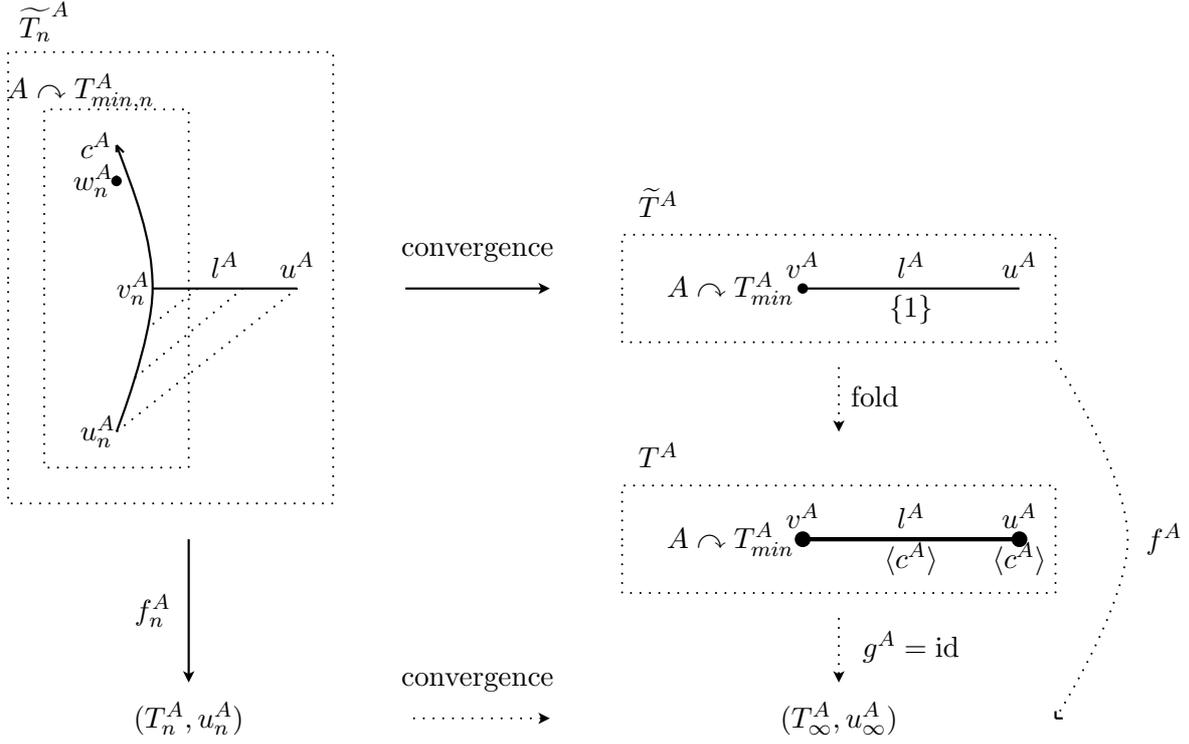}
\caption{The situation in Case 1 of the proof of Lemma \ref{reduction}.}
\label{fig-1-shape-2}
\end{center}
\end{figure}

We now prove that the pointed trees $(T^A_n,u_n^A)$ converge to $(T^A,u^A)$. Lemma \ref{Guirardel-reduction} implies that the trees $(\widetilde{T_n}^A,u^A)$ converge to $(\widetilde{T}^A,u^A)$. For all $n\in\mathbb{N}$, there is a $1$-Lipschitz $G$-equivariant map $f^A_n:(\widetilde{T_n}^A,u^A)\to (T^A_n,u_n^A)$. This implies that for all $g\in G$, we have $d_{T_n^A}(u_n^A,gu_n^A)\le d_{\widetilde{T_n}^A}(u^A,gu^A)$. Therefore, up to possibly passing to a subsequence, the pointed trees $(T_n^A,u_n^A)$ converge to a pointed tree $(T_{\infty}^A,u_{\infty}^A)$ in the Gromov--Hausdorff equivariant topology, that is minimal in the sense of pointed $G$-trees. Proposition \ref{pointed-convergence} shows that there exists a $1$-Lipschitz map $f^A:(\widetilde{T}^A,u^A)\to (\overline{T^A_{\infty}},u^A_{\infty})$, where $\overline{T^A_{\infty}}$ denotes the metric completion of $T^A_{\infty}$. We will show that $f^A$ factors through a map $g^A:(T^A,u^A)\to (\overline{T^A_{\infty}},u^A_{\infty})$, and that $g^A$ is an isometry between $(T^A,u^A)$ and $(T^A_{\infty},u^A_{\infty})$. This will imply that the pointed trees $(T^A_n,u_n^A)$ converge to $(T^A,u^A)$.

We first notice that for all $n\in\mathbb{N}$, the map $f^A_n$ is an isometry in restriction to $T^{A}_{min,n}$. By taking limits, this implies that the $A$-minimal subtree $T^{A}_{min}$ of $T^A$ isometrically embeds into $T^A_{\infty}$. In addition, for all $n\in\mathbb{N}$, the point $u_n^A$ belongs to the axis of $c^A$ in $T_n$. By definition of the equivariant Gromov--Hausdorff topology on the set of pointed $(G,\mathcal{F})$-trees, this implies that $c^A$ fixes $u^A_{\infty}$ in $T_{\infty}^A$. Similarly, the element $c^A$ fixes all points of the image of $[u^A,v^A]$ in $T^A$. Therefore, the map $f^A$ factors through a map $g^A:(T^A,u^A)\to (\overline{T^A_{\infty}},u^A_{\infty})$. As $T^{A}_{min}$ isometrically embeds into $T^A_{\infty}$, the map $g^A$ can only decrease the length of the segment $[u^A,v^A]$, and fold this segment over a subarc of $[v^A,w^A]$. 

Let $g\in A$ be an element that is hyperbolic in $T^A_{min}$ (we recall that we have assumed $T^A_{min}$ not to be reduced to a point), such that $d_{T^A}(v^A,gv^A)=d_{T^A}(w^A,gw^A)+2d_{T^A}(v^A,w^A)$. In particular, we have $d_{T^A}(u^A,gu^A)=d_{T^A}(w^A,gw^A)+2l^A+2d_{T^A}(v^A,w^A)$. Using the fact that we folded in a direction that did not contain $w_n$, together with the definition of the equivariant Gromov--Hausdorff topology, we get that the distance $d_{T_n^A}(u_n^A,gu_n^A)$ gets arbitrarily close to $d_{T_n^A}(w_n^A,gw_n^A)+2l^A+2d_{T^A}(v^A,w^A)$ as $n$ tends to $+\infty$, so $d_{T_{\infty}^A}(u_{\infty}^A,gu_{\infty}^A)=d_{T^A}(w^A,gw^A)+2l^A+2d_{T^A}(v^A,w^A)$. This implies that $g^A$ is an isometry from $(T^A,u^A)$ to $(T^A_{\infty},u^A_{\infty})$, and we are done.

If $T^A_{min}$ is reduced to a point, then it can be approximated by a sequence $(T^A_{min,n})_{n\in\mathbb{N}}$ of Grushko $(A,\mathcal{F}_{|A})$-trees, where all edge lengths are equal to $\frac{1}{n}$, which we choose to be all homothetic to each other. We also choose two distinct constant sequences $v_n^A$ and $w_n^A$ in the trees $T^A_{min,n}$, and construct the trees $T_n^A$ as above. Let $g\in A$ be any element such that $d_{T_n^A}(v_n^A,gv_n^A)=d_{T_n^A}(w_n^A,gw_n^A)+2d_{T_n^A}(v_n^A,w_n^A)$ for all $n\in\mathbb{N}$. Arguing similarly as above, we get that $d_{T^A_{\infty}}(u^A_{\infty},gu_{\infty}^A)=2l^A$. This again implies that the map $g^A$ defined as above is an isometry.
\\
\\
\textit{Case 2}: The splitting $S$ is of the form $G=C\ast$.
\\
\\
The vertex tree $T^C$ of $\mathcal{G}$ may fail to be minimal. We denote by $u_1$ and $u_2$ two points in $T^C$ in the orbits of the attaching points (the points $u_1$ and $u_2$ may belong to the same $G$-orbit). We denote by $v_1$ and $v_2$ their projections to the closure $\overline{T^C_{min}}$ of the $C$-minimal subtree of $T$. One of the following cases occurs.
\\
\\
\textit{Case 2.1}: The segments $[u_1,v_1]$ and $[u_2,v_2]$ are nondegenerate, and their stabilizers are nontrivial and nonconjugate in $C$.
\\
In other words, the tree $T$ splits as a graph of actions that has the shape displayed on Figure \ref{fig-approx-21}, where $l_1,l_2>0$, and the stabilizers $\langle c_1\rangle$ and $\langle c_2\rangle$ are nonconjugate. We allow the case where $v_1$ and $v_2$ belong to the same $G$-orbit. For all $i\in\{1,2\}$, we let $w_i$ be such that $[v_i,w_i]$ is the maximal arc fixed by $c_i$ in ${T_{min}^C}$, if this arc is nondegenerate, and we let $w_i$ be any point distinct from $v_i$ otherwise (as in Case 1, one has to slightly adapt the argument when $T_{min}^C$ is reduced to a point). Let $\widetilde{T}^C$ be the tree obtained from $T^C$ by replacing the edges $[u_1,v_1]$ and $[u_2,v_2]$ by edges of the same length with trivial stabilizer. For all $i\in\{1,2\}$, let $v_{n,i}$ (resp. $w_{n,i}$) be an approximation of $v_i$ (resp. $w_i$) in an approximation of ${T^C_{min}}$. We can assume $v_{n,i}$ to belong to the translation axis of $c_i$. Let $(\widetilde{T_n}^C,u_{1},u_{2})$ be the approximation of $(\widetilde{T}^C,u_1,u_2)$ obtained from an approximation of $T^C_{min}$ by adding an edge of length $l_1$ (resp. $l_2$) with trivial stabilizer at $v_{n,1}$ (resp. $v_{n,2}$). Let $T_n^C$ be the tree obtained from $\widetilde{T_n}^C$ by $G$-equivariantly fully folding the edge $[u_{i},v_{n,i}]$ along the axis of $c_i$, in a direction that does not contain $w_{n,i}$, for all $i\in\{1,2\}$. We denote by $f_n^C:\widetilde{T_n}^C\to T_n^C$ the corresponding morphism. Arguing as in Case 1, one shows that the trees $(T_n^C,f_n^C(u_{1}),f_n^C(u_{2}))$ converge to $(T^C,u_1,u_2)$. Let now $T_n$ be the tree obtained by replacing $(T^C,u_1,u_2)$ by its approximation $(T_n^C,f_n^C(u_{1}),f_n^C(u_{2}))$ in the graph of actions $\mathcal{G}$. Lemma \ref{Guirardel-reduction} implies that the trees $T_n$ converge to $T$.

\begin{figure}
\begin{center}
\input{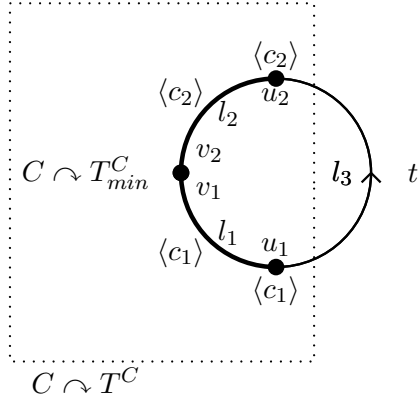}
\caption{The splitting of $T$ as a graph of actions in Cases 2.1, 2.2 and 2.5 of the proof of Lemma \ref{reduction}.}
\label{fig-approx-21}
\end{center}
\end{figure}

~
\\
\noindent \textit{Case 2.2}: The segments $[u_1,v_1]$ and $[u_2,v_2]$ are nondegenerate and have nontrivial stabilizers that are conjugate in $C$, and no two nondegenerate subsegments of $[u_1,v_1]$ and $[u_2,v_2]$ belong to the same $G$-orbit.
\\
Again, the tree $T$ splits as a graph of actions that has the shape displayed on Figure~\ref{fig-approx-21}, where this time the groups $\langle c_1\rangle$ and $\langle c_2\rangle$ are conjugate. Up to a good choice of the stable letter $t$, we can assume that $\langle c_1\rangle=\langle c_2\rangle$. As tripod stabilizers are trivial in $T$, the segment $[v_1,v_2]$ is the maximal arc fixed by $c_1$ in $T^C_{min}$. Again, we let $\widetilde{T}^C$ be the tree obtained from $T^C$ by replacing the edges $[u_1,v_1]$ and $[u_2,v_2]$ by edges of the same length with trivial stabilizer. For all $i\in\{1,2\}$, let $v_{n,i}$ be an approximation of $v_i$ in an approximation of $\overline{T^C_{min}}$ by minimal Grushko $(C,\mathcal{F}_{|C})$-trees, which we can assume to belong to the translation axis of $c_1$. Let $(\widetilde{T_n}^C,u_{1},u_{2})$ be the approximation of $(\widetilde{T}^C,u_1,u_2)$ obtained from an approximation of $T^C_{min}$ by adding an edge of length $l_1$ (resp. $l_2$) with trivial stabilizer at $v_{n,1}$ (resp. $v_{n,2}$). The tree $T_n^C$ is then obtained from $\widetilde{T_n}^C$ by $G$-equivariantly fully folding the edges $[u_{1},v_{n,1}]$ and $[u_{2},v_{n,2}]$ along the axis of $c_1$ in opposite directions. The folding directions should not contain the segment $[v_{n,1},v_{n,2}]$, in case this segment is nondegenerate. Again, denoting by $f_n^C:\widetilde{T_n}^C\to T_n^C$ the corresponding morphism, the trees $(T_n^C,f_n^C(u_{1}),f_n^C(u_{2}))$ converge to $(T^C,u_1,u_2)$. The trees $T_n$ obtained by replacing $(T^C,u_1,u_2)$ by $(T_n^C,f_n^C(u_1),f_n^C(u_2))$ in $\mathcal{G}$ converge to $T$.
\\
\\
\textit{Case 2.3}: Some nondegenerate subsegments of $[u_1,v_1]$ and $[u_2,v_2]$ belong to the same $G$-orbit, and their common stabilizer is nontrivial.
\\
Using the fact that tripod stabilizers in $T$ are trivial, we can assume that $v_1=v_2$ (and we let $v:=v_1=v_2$), and that $[u_1,v]\subseteq [u_2,v]$. The tree $T$ splits as a graph of actions that has the form displayed on Figure \ref{fig-approx-23}. We let $w$ be such that $[v,w]$ is the maximal arc fixed by $c$ in ${T_{min}^C}$ if this arc is nondegenerate, and choose any $w\neq v$ otherwise (as in Case 1, one has to slightly adapt the argument if $T^C_{min}$ is reduced to a point). Let $\widetilde{T}^C$ be the tree obtained from $T$ by replacing the segment $[u_2,v]$ by a segment of same length $l_1+l_2$ with trivial stabilizer. Let $v_{n}$ (resp. $w_{n}$) be an approximation of $v$ (resp. $w$) in an approximation $Y_n$ of $\overline{T^C_{min}}$. We can assume $v_n$ to belong to the translation axis of $c$ in $Y_n$. Let $(\widetilde{T_n}^C,u_2)$ be the approximation of $(\widetilde{T}^C,u)$ obtained from $Y_n$ by adding an edge of length $l_1+l_2$ with trivial stabilizer at $v_{n}$. The tree $T_n^C$ is then obtained from $\widetilde{T_n}^C$ by $G$-equivariantly fully folding the edge $[u_2,v_{n}]$ along the axis of $c$, in a direction that does not contain $w_{n}$. Denoting by $f_n^C:\widetilde{T_n}^C\to T_n^C$ the corresponding morphism, the trees $(T_n^C,f_n^C(u_1),f_n^C(u_2))$ converge to $(T^C,u_1,u_2)$. Again, the trees $T_n$ obtained by replacing $(T^C,u_1,u_2)$ by $(T_n^C,f_n^C(u_1),f_n^C(u_2))$ in $\mathcal{G}$ converge to $T$.

\begin{figure}
\begin{center}
\input{approx-23.pst}
\caption{The splitting of $T$ as a graph of actions in Case 2.3 of the proof of Lemma \ref{reduction}.}
\label{fig-approx-23}
\end{center}
\end{figure}

~
\\
\textit{Case 2.4}: Some nondegenerate subsegments of $[u_1,v_1]$ and $[u_2,v_2]$ belong to the same $G$-orbit, and they have trivial stabilizer.
\\
Then $T$ splits as a graph of actions of the form displayed on Figure \ref{fig-approx-24}. This case may be viewed as a particular case of Case 1.
\\
\\
\textit{Case 2.5}: One of the subsegments $[u_1,v_1]$ or $[u_2,v_2]$ is degenerate. 
\\
This case is treated in a similar way as Case 2.1, and left to the reader.
\end{proof}

\begin{figure}
\begin{center}
\input{approx-24.pst}
\caption{The splitting of $T$ as a graph of actions in Case 2.4 of the proof of Lemma \ref{reduction}.}
\label{fig-approx-24}
\end{center}
\end{figure}

\subsection{Dynamical decomposition of a geometric very small $(G,\mathcal{F})$-tree}\label{sec-dyn}

Every geometric very small $(G,\mathcal{F})$-tree splits as a graph of actions, which has the following description.

\begin{prop}\label{goa-geom}
Any very small geometric $(G,\mathcal{F})$-tree $T$ splits as a graph of actions $\mathcal{G}$, where for each nondegenerate vertex action $Y$, with vertex group $G_Y$, and attaching points $v^1,\dots,v^k$ fixed by subgroups $H^1,\dots,H^k$, either 
\begin{itemize}
\item the tree $Y$ is an arc containing no branch point of $T$ except at its endpoints, or  
\item the group $G_Y$ is the fundamental group of a $2$-orbifold with boundary $\Sigma$ holding an arational measured foliation, and $Y$ is dual to $\widetilde{\Sigma}$, or
\item there exists a Lipschitz approximation of $Y$ by pointed Grushko $(G_Y,\mathcal{F}_{|G_Y})$-trees $(Y_n,(v^1_n,\dots,v^k_n))$, whose quotient volumes converge to $0$, such that for all $n\in\mathbb{N}$, there exists a morphism $f_n:Y_n\to Y$, with $f_n(v^i_n)=v^i$ for all $i\in\{1,\dots,k\}$, and $v^i_n$ is fixed by $H^i$ for all $i\in\{1,\dots,k\}$ and all $n\in\mathbb{N}$.
\end{itemize}
\end{prop}

We call $\mathcal{G}$ the \emph{dynamical decomposition} of $T$, it is determined by $T$. Vertex actions of the third type are called \emph{exotic}. In case no vertex action is exotic, we say that $T$ is \emph{of surface type}.

The proof of Proposition \ref{goa-geom} goes as follows. Let $T$ be a geometric $(G,\mathcal{F})$-tree, and let $K\subseteq T$ be a finite subtree such that $T=T_{K}$. Let $A:=\Sigma/G$, where $\Sigma$ is the band complex defined in Section \ref{sec-pg-tree}. Let $A^{\ast}$ be the complement of the singular set in $F\subseteq A$ (where $F$ is the compact forest defined in Section \ref{sec-systems}), endowed with the restriction of the foliation of $\Sigma$. Let $C^{\ast}\subseteq A^{\ast}$ be the union of the leaves of $A^{\ast}$ which are closed but not compact. The \emph{cut locus} of $A$ is defined as $C:=C^{\ast}\cup \text{Sing}$. The set $A\smallsetminus C$ is a union of finitely many open sets $U_1,\dots,U_p$, which are unions of leaves of $A^{\ast}$. By a classical result of Imanishi \cite{Ima79}, see also \cite[Section 3]{GLP94}, for all $i\in\{1,\dots,p\}$, either every leaf of $U_i$ is compact, or else every leaf of $U_i$ is dense in $U_i$. Notice that each component $U_i$ is again dual to a finite system of isometries.

As noticed in \cite[Propositions 1.25 and 1.31]{Gui08}, Imanishi's theorem provides a transverse covering of $T$ in the following way. Let $\widetilde{C}$ be the lift of the cut locus to $\Sigma$. Given a component $U$ of $\Sigma\smallsetminus\widetilde{C}$, we let $T_U$ be the tree dual to the foliated $2$-complex $\overline{U}$, i.e. the leaf space made Hausdorff of $\overline{U}$. Then the family $\{T_U\}_U$ is a transverse covering of $T$. Each $T_U$ is either an arc (in the case where every leaf of the image of $U$ in $A$ is compact) or has dense orbits (in the case where all leaves of the image of $U$ in $A$ are dense). Associated to this transverse covering of $T$ is a graph of actions, whose vertex groups are finitely generated, and whose vertex actions are dual to foliated $2$-complexes. In addition, arc stabilizers in the vertex actions with dense orbits are trivial (Proposition~\ref{dense-arcs}). Therefore, we can apply \cite[Proposition A.6]{Gui08} to each of the vertex actions with dense orbits. This provides a classification of vertex actions with dense orbits into three types (\emph{axial}, \emph{surface} and \emph{exotic}). As our system of partial isometries on $F$ has independent generators (Lemma \ref{lemma-genind}), all vertex actions with dense orbits of the decomposition are dual to finite systems of isometries with independent generators. This excludes the axial case, see \cite[Proposition 3.4]{Gab97}. The existence of the Lipschitz approximation with the required properties in the exotic case was proved by Guirardel in \cite[Proposition 7.2]{Gui98}, using a pruning and narrowing argument.
\\
\\
\indent We finish this section by mentioning a consequence of Proposition \ref{goa-geom}, that will turn out to be useful in \cite{Hor14-8}.

\begin{lemma}\label{surface-type}
Let $T$ be a small, minimal $(G,\mathcal{F})$-tree. If there exists a subgroup $H\subseteq G$ that is elliptic in $T$, and not contained in any proper $(G,\mathcal{F})$-free factor, then $T$ is geometric of surface type.
\end{lemma}

\begin{proof}
If $T$ is geometric, this follows from Proposition \ref{goa-geom}, so we assume that $T$ is nongeometric. Up to replacing $H$ by the point stabilizer of $T$ in which it is contained, we can assume that $\text{rk}_K(H)<+\infty$. Theorem \ref{approx-by-geom-2} lets us approximate $T$ by a sequence $(T_n)_{n\in\mathbb{N}}$ of small, minimal geometric $(G,\mathcal{F})$-trees, in which $H$ is elliptic. The trees $T_n$ come with morphisms onto $T$. As $T$ is nongeometric, Corollary \ref{stab-isom} ensures that the trees $T_n$ contain an edge with trivial stabilizer. This implies that $H$ is contained in a proper $(G,\mathcal{F})$-free factor, a contradiction. 
\end{proof}

\subsection{Trees of surface type}

Let $T$ be a very small geometric $(G,\mathcal{F})$-tree of surface type (where we recall the definition from the paragraph below Proposition \ref{goa-geom}). Let $\mathcal{G}$ be the dynamical decomposition of $T$, and let $S$ be the skeleton of the corresponding transverse covering. It follows from Proposition \ref{goa-geom} that there are three types of vertices in $S$, namely: vertices of \emph{surface type}, of \emph{arc type}, and vertices associated to nontrivial intersections between the trees of the transverse covering, which we call vertices of \emph{trivial type}. All edge stabilizers in $S$ are cyclic (possibly finite or peripheral). Indeed, stabilizers of edges adjacent to vertices of surface type are either trivial, or they are cyclic, and represent boundary curves or conical points of the associated orbifold. Both edges adjacent to vertices of arc type have the same stabilizer, equal to the stabilizer of the corresponding arc in $T$, which is cyclic because $T$ is very small.

\begin{de}\label{de-unused}
Let $T$ be a very small geometric $(G,\mathcal{F})$-tree of surface type, and let $\sigma$ be a compact $2$-orbifold arising in the dynamical decomposition $\mathcal{G}$ of $T$. Let $g\in G$ be an element represented by a boundary curve of $\sigma$. We say that $g$ is \emph{used} in $T$ if either $g$ is peripheral, or $g$ is conjugate into some adjacent edge group of $\mathcal{G}$. Otherwise $g$ is \emph{unused} in $T$.
\end{de}

\begin{prop}\label{outer-limits}
Let $T$ be a minimal, very small, geometric $(G,\mathcal{F})$-tree of surface type. Then either there exists an unused element in $T$, or $T$ splits as a $(G,\mathcal{F})$-graph of actions over a one-edge $(G,\mathcal{F})$-free splitting.
\end{prop}

Our proof of Proposition \ref{outer-limits} is based on the following lemma. Given a group $G$, and a family $Y$ of subgroups of $G$, we say that $G$ \emph{splits as a free product relatively to $Y$} if there exists a splitting of the form $G=A\ast B$, such that every subgroup in $Y$ is conjugate into either $A$ or $B$.

\begin{lemma} \label{unfold}
Let $G$ be a countable group, and let $\mathcal{F}$ be a free factor system of $G$. Let $T$ be a minimal, simplicial $(G,\mathcal{F})$-tree, whose edge stabilizers are all cyclic and nontrivial (they may be finite or peripheral). Then there exists a vertex $v$ in $\mathcal{G}$, such that $G_v$ splits as a free product relative to incident edge groups and subgroups in $\mathcal{F}_{|G_v}$. 
\end{lemma}

Let $T,T'$ be two simplicial $(G,\mathcal{F})$-trees. A map $f:T\to T'$ is a \emph{$G$-equivariant edge fold} (or simply a \emph{fold}) if there exist two edges $e_1,e_2\subseteq T$, incident to a common vertex in $T$, such that $T'$ is obtained from $T$ by $G$-equivariantly identifying $e_1$ and $e_2$, and $f:T\to T'$ is the quotient map. A fold $f:T\to T'$ is determined by the orbit of the pair of edges $(e_1,e_2)$ identified by $f$. We say that $f$ is 

\begin{itemize}
\item of \emph{type $1$} if $e_1$ and $e_2$ belong to distinct $G$-orbits of oriented edges in $T$, and both $e_1$ and $e_2$ have nontrivial stabilizer, and
\item of \emph{type $2$} if $e_1$ and $e_2$ belong to distinct $G$-orbits of oriented edges in $T$, and either $e_1$ or $e_2$ (or both) has trivial stabilizer in $T$, and
\item of \emph{type $3$} if $e_1$ and $e_2$ belong to the same $G$-orbit of oriented edges in $T$.
\end{itemize}

Assume that $f:T\to T'$ is a fold. We note that if $H\subseteq G$ is a subgroup of $G$ that fixes an edge $e_1\subseteq T'$, and $\widetilde{e_1}$ is an edge in the $f$-preimage of $e_1$ in $T$, then $H$ fixes an extremity of $\widetilde{e_1}$. We start by making the following observation.

\begin{lemma}\label{obs}
Let $T$ and $T'$ be two simplicial $(G,\mathcal{F})$-trees with cyclic edge stabilizers. Assume that $T'$ is obtained from $T$ by performing a fold $f$ of type $2$ or $3$. If $e'_1$ and $e'_2$ are two edges of $T$ that are identified by $f$, then either $e'_1$ or $e'_2$ (or both) has trivial stabilizer.
\qed
\end{lemma}

\begin{proof}[Proof of Lemma \ref{unfold}]
Let $T_0$ be any Grushko $(G,\mathcal{F})$-tree. All point stabilizers in $T_0$ are elliptic in $T$, so up to possibly collapsing some edges in $T_0$, and subdividing edges of $T_0$, there exists a simplicial map $f:T_0\to T$ (i.e. sending vertices to vertices and edges to edges). By \cite[3.3]{Sta83}, the map $f$ can be decomposed as a sequence of $G$-equivariant edge folds $f_i:T_i\to T_{i+1}$.  

We can assume that along the folding sequence, we always perform edge folds of type $1$ before performing edge folds of type $2$ and $3$, and we always perform edge folds of type $2$ before edge folds of type $3$, for as long as possible. This is possible because the number of orbits of edges decreases when performing a fold of type $1$ or $2$. 

We claim that we can also assume that folds are maximal in the following sense: if $g$ fixes an edge $e$ in $T$, then we never identify a preimage $e'$ of $e$ with a translate of the form $g^ke'$ without identifying it with $ge'$. Assume otherwise, and let $i$ be the first time at which a non-maximal fold occurs. Let $e'$ be a preimage of $e$ in $T_i$, having a vertex $y$ stabilized by $g^k$, so that we fold $e'$ and $g^ke'$ when passing from $T_i$ to $T_{i+1}$. By our choice of $i$ and the fact that edge stabilizers in $T$ are cyclic, the edge $e'$ has trivial stabilizer in $T_i$ (in particular, the fold performed from $T_i$ to $T_{i+1}$ is not of type $1$). The element $g$ is also elliptic in $T_i$. We let $x$ be the point closest to $y$ that is fixed by $g$ in $T_i$. If $x=y$, then we could choose to identify $e'$ and $ge'$ when passing from $T_i$ to $y$. Otherwise, all edges in the segment $[x,y]\subseteq T_i$ are stabilized by $g^k$. Our choice of $i$ implies that the stabilizer of their images in $T$ is equal to $g^k$ (and not to any proper root of $g^k$). Since the image of $e'$ in $T$ is stabilized by $g$, this implies that we can find two consecutive edges on the segment $[x,y]$ that are identified in $T$. This shows that one could perform a fold of type $1$ on the tree $T_i$, contradicting our choice of folding path. 

Let $T_k$ be the last tree along the folding sequence that contains an edge with trivial stabilizer. The fold $f_k$ is either of type $2$ or of type $3$. It identifies an edge $e_k$ of $T_k$ with trivial stabilizer with some translate $ge_k$ with $g\in G$ (although the pair $(e_k,ge_k)$ might not be the defining pair of the edge fold). We can assume $\langle g\rangle$ to be maximal in the following sense: if $g$ is of the form $h^l$ with $h\in G$ and $l>1$, then $e_k$ is not identified with $he_k$. We claim that $T_{k+1}=T$. 

We postpone the proof of the claim to the next paragraph, and first explain how to derive the lemma from this claim. Let $v_k$ be the vertex of $e_k$ such that $gv_k=v_k$, and $v'_k$ be the other vertex of $e_k$. Notice that $G_{v'_k}$ is nontrivial: otherwise, all edges in $T_k$ adjacent to $v'_k$ would have trivial stabilizer, and there would be at least three distinct $G$-orbits of such edges. Since the fold $f_k$ involves at most two orbits of edges, there would be an edge with trivial stabilizer in $T_{k+1}$, contradicting the definition of $T_{k+1}$. If $f_k$ is a fold of type $3$, defined by the pair $(e_k,ge_k)$, then the vertex $f_k(v'_k)$ satisfies the conclusion of the lemma. Indeed, we have $G_{f_k(v'_k)}=G_{v'_k}\ast\langle g\rangle$. If $f_k$ is a fold of type $2$, then $f_k$ identifies $e_k$ with an edge $e'_k$ with nontrivial stabilizer, because otherwise $T_{k+1}$ would contain an edge with trivial stabilizer. Denote by $v''_k$ the vertex of $e'_k$ distinct from $v_k$. If $v'_k$ and $v''_k$ do not belong to the same $G$-orbit, then we have a nontrivial splitting $G_{f_k(v'_k)}=G_{v'_k}\ast G_{v''_k}$.   If $v'_k$ and $v''_k$ belong to the same $G$-orbit, then we have a nontrivial splitting $G_{f_k(v'_k)}=G_{v''_k}\ast\mathbb{Z}$. In both cases, this splitting is relative to incident edge groups and to $\mathcal{F}_{|G_v}$, because all trees along the folding sequence are $(G,\mathcal{F})$-trees.

We now prove the above claim that $T_{k+1}=T$. Assume towards a contradiction that $T_{k+1}\neq T$. It follows from our choice of folding path, and the fact that $T$ has cyclic edge stabilizers, that all possible folds in $T_{k+1}$ identify two edges $e_1$ and $e_2$ at $f_{k}(v'_k)$ in distinct $G$-orbits, and $e_1$ and $e_2$ have the same nontrivial stabilizer $H$ in $T_{k+1}$. Let $\widetilde{e_1}$ (resp. $\widetilde{e_2}$) be an edge in $T_k$ in the $f_k$-preimage of $e_1$ (resp. $e_2$). The group $H$ fixes an extremity of both $\widetilde{e_1}$ and $\widetilde{e_2}$. If $\widetilde{e_1}$ and $\widetilde{e_2}$ were disjoint, then $H$ would fix the segment between them, a contradiction (Lemma \ref{obs}). Therefore, the edges $\widetilde{e_1}$ and $\widetilde{e_2}$ are adjacent in $T_k$. By our choice of folding path, at least one of them, say $\widetilde{e_1}$, has trivial stabilizer (otherwise we could perform a fold of type $1$ in $T_k$ identifying $\widetilde{e_1}$ and $\widetilde{e_2}$), and hence belongs to the $G$-orbit of $e_k$. Since it is possible to fold $\widetilde{e_1}$ and $\widetilde{e_2}$ in $T_k$, the fold $f_k$ is of type $2$. Hence $f_k$ identifies $\widetilde{e_1}$ with an edge $e_3$ whose stabilizer is equal to $H$. Then $e_3$ is adjacent to $e_2$, and is identified with $e_2$ in $T$, so we could have performed a fold of type $1$ in $T_k$, a contradiction. 
\end{proof}

\begin{proof}[Proof of Proposition \ref{outer-limits}]
Let $S$ be the skeleton of the dynamical decomposition of $T$. If $S$ is reduced to a point, then $T$ is dual to a minimal measured foliation on a compact $2$-orbifold $\sigma$. Some boundary component $c$ of $\sigma$ represents a nonperipheral conjugacy class. Indeed, all boundary components of a compact $2$-orbifold cannot be elliptic in a common free splitting of the fundamental group of the orbifold. Then $c$ is an unused element in $T$.

Now assume that $S$ is a nontrivial minimal $(G,\mathcal{F})$-tree. If $S$ contains an edge with trivial stabilizer, then this edge defines a $(G,\mathcal{F})$-free splitting, and $T$ splits as a graph of actions over this splitting by \cite[Lemma 4.7]{Gui04}. Otherwise, let $v$ be a vertex of $S$ provided by Lemma \ref{unfold}.

We note that the vertex $v$ cannot be of arc type, because a vertex of arc type has a stabilizer equal to the stabilizer of the incident edges. If $v$ is of surface type, it is associated to a compact $2$-orbifold $\sigma$, equipped with a minimal measured foliation. The fundamental group of $\sigma$ splits as a free product relatively to incident edge groups and to subgroups in $\mathcal{F}_{|\pi_1(\sigma)}$. This implies the existence of an unused element in $T$, otherwise $\pi_1(\sigma)$ would split as a free product in which all boundary components of $\sigma$ are elliptic. 

If $v$ is of trivial type, then we get a one-edge $(G,\mathcal{F})$-free splitting $S_0$ that is compatible with $S$, by splitting the vertex group $G_v$. Associated to each vertex of $S_0$ with vertex group $G_{v'}$ is a geometric (possibly trivial) $G_{v'}$-action $T_{v'}$. The tree $T$ splits as a graph of actions over $S_0$, with the trees $T_{v'}$ as vertex actions.
\end{proof}

\subsection{Approximating very small geometric $(G,\mathcal{F})$-trees by Grushko $(G,\mathcal{F})$-trees}

\begin{proof}[Proof of Theorems \ref{cv-closure} and \ref{Lip-approx}]
Let $T$ be a very small, minimal $(G,\mathcal{F})$-tree. We want to show that we can approximate $T$ by a sequence of minimal Grushko $(G,\mathcal{F})$-trees, and that the approximation can be chosen to be a Lipschitz approximation, by trees whose quotient volumes converge to the quotient volume of $T$, if $T$ has trivial arc stabilizers. We will argue by induction on $\text{rk}_K(G,\mathcal{F})$. The claim holds true when $\text{rk}_K(G,\mathcal{F})=1$, so we assume that $\text{rk}_K(G,\mathcal{F})\ge 2$. The claim also holds true if $T$ is reduced to a point. Thanks to Theorem \ref{approx-by-geom}, we can assume $T$ to be geometric. (For the statement about quotient volumes, see the argument at the beginning of Section \ref{sec-reduction}). By Proposition~\ref{goa-geom}, the tree $T$ decomposes as a graph of actions whose vertex actions are either arcs, or of surface or exotic type. Proposition \ref{goa-geom} also enables us to approximate all exotic vertex actions. Using Lemmas \ref{Guirardel-reduction} and \ref{reduction-2}, we can therefore reduce to the case where $T$ is a tree of surface type (notice that all edges of the decomposition as a graph of actions whose stabilizer is noncyclic, or nontrivial and peripheral, have length $0$).
 
First assume that there exists an unused element $c$ in $T$, corresponding to a boundary curve in a minimal orbifold $\sigma$ of the dynamical decomposition. One can then narrow the band complex by width $\delta>0$ from $c$ to get a Lipschitz approximation of $T$: this is done by cutting a segment on $\Sigma$ of length $\delta$ (arbitrarily small) transverse to the boundary curve $c$ and to the foliation. In this way, all leaves of the foliations become segments (half-leaves of the original foliation on $\sigma$ are dense), so the tree dual to the foliated complex by which the minimal foliation on $\sigma$ has been replaced is simplicial. By choosing $\delta>0$ arbitrarily close to $0$, we can ensure the volume of this tree to be arbitrarily small. In the new band complex obtained in this way, the orbifold $\sigma$ has therefore been replaced by a simplicial component, with trivial edge stabilizers. 
 
We thus reduce to the case where no element of $G$ is unused in $T$. Lemma \ref{outer-limits} thus ensures that $T$ splits as a $(G,\mathcal{F})$-graph of actions over a one-edge $(G,\mathcal{F})$-free splitting, and we can conclude by induction, using Lemmas \ref{reduction-2} and \ref{reduction}.
\end{proof}

\section{Tame $(G,\mathcal{F})$-trees}\label{sec-good}

We finish this paper by introducing another class of $(G,\mathcal{F})$-trees, larger than the class of very small $(G,\mathcal{F})$-trees, which we call \emph{tame} $(G,\mathcal{F})$-trees. This class will turn out to provide the right setting for carrying out our arguments in \cite{Hor14-6} to describe the Gromov boundary of the graph of cyclic splittings of $(G,\mathcal{F})$.

\begin{de}
A minimal $(G,\mathcal{F})$-tree is \emph{tame} if it is small, and has finitely many orbits of directions at branch points.
\end{de}

There exist small $(G,\mathcal{F})$-actions that are not tame. A typical example is the following: a sequence of splittings of $F_2=\langle a,b\rangle$ of the form $F_2=(\langle a\rangle\ast_{\langle a^2\rangle}\langle a^2\rangle\ast_{\langle a^4\rangle}\dots\ast_{\langle a^{2^n}\rangle}\langle a^{2^n}\rangle)\ast \langle b\rangle$, in which the edge with stabilizer generated by $a^{2^k}$ has length $\frac{1}{2^k}$, converges to a small $F_2$-tree with infinitely many orbits of branch points. 

By the discussion in Section \ref{sec-Levitt}, the class of tame $(G,\mathcal{F})$-trees is the right class of trees in which Levitt's decomposition makes sense. 

\begin{theo} (Levitt \cite[Theorem 1]{Lev94})\label{Levitt-good}
Let $G$ be a countable group, and let $\mathcal{F}$ be a free factor system of $G$. Then every tame $(G,\mathcal{F})$-tree splits uniquely as a graph of actions, all of whose vertex trees have dense orbits for the action of their stabilizer, such that the Bass--Serre tree of the underlying graph of groups is small, and all its edges have positive length.
\end{theo}

However, the above example of a small $(G,\mathcal{F})$-tree that is not tame shows that the space of tame $(G,\mathcal{F})$-trees is not closed.  We will describe conditions under which a limit of tame trees is tame. It will be of interest to introduce yet another class of $(G,\mathcal{F})$-trees. The equivalences in the definition below are straightforward.

\begin{de}
Let $k\in\mathbb{N}^{\ast}$. A small minimal $(G,\mathcal{F})$-tree is \emph{$k$-tame} if one of the following equivalent conditions occurs.
\begin{itemize}
\item For all nonperipheral $g\in G$ and all arcs $I\subseteq T$, if $\langle g\rangle\cap \text{Stab}(I)$ is nontrivial, then its index in $\langle g\rangle$ divides $k$. 
\item For all nonperipheral $g\in G$, and all $l\ge 1$, we have $\text{Fix}(g^l)\subseteq\text{Fix}(g^k)$.
\item For all nonperipheral $g\in G$, and all $l\ge 1$, we have $\text{Fix}(g^{kl})=\text{Fix}(g^{k})$. 
\end{itemize}
\end{de}

Notice in particular that $1$-tame $(G,\mathcal{F})$-trees are those $(G,\mathcal{F})$-trees in which all arc stabilizers are either trivial, or maximally-cyclic and nonperipheral.  

\begin{prop}
For all $k\in\mathbb{N}^{\ast}$, the space of $k$-tame $(G,\mathcal{F})$-trees is closed in the space of small minimal $(G,\mathcal{F})$-trees.
\end{prop}

\begin{proof}
Let $T$ be a small, minimal $(G,\mathcal{F})$-tree, and let $(T_n)_{n\in\mathbb{N}}$ be a sequence of $k$-tame $(G,\mathcal{F})$-trees that converges to $T$. Let $g\in G$, and assume that there exists $l\ge 1$ such that $g^l$ fixes a nondegenerate arc $[a,b]\subseteq T$. If $g$ is hyperbolic in $T_n$ for infinitely many $n\in\mathbb{N}$, then Lemma \ref{lemma-good-2} implies that $g$ fixes $[a,b]$. We can therefore assume that for all $n\in\mathbb{N}$, the fixed point set $I_n$ of $g^l$ is nonempty. Let $a_n$ (resp. $b_n$) be an approximation of $a$ (resp. $b$) in $T_n$. Since $d_{T_n}(g^la_n,a_n)$ and $d_{T_n}(g^lb_n,b_n)$ both converge to $0$, the distances $d_{T_n}(a_n,I_n)$ and $d_{T_n}(b_n,I_n)$ both converge to $0$. As $T_n$ is $k$-tame, this implies that the distances of both $a_n$ and $b_n$ to the fixed point set of $g^k$ in $T_n$ converge to $0$. So both $d_{T_n}(g^{k}a_n,a_n)$ and $d_{T_n}(g^{k}b_n,b_n)$ converge to $0$, and therefore $g^{k}$ fixes $[a,b]$ in $T$. This shows that $T$ is $k$-tame.
\end{proof}

\begin{prop}\label{good}
A minimal $(G,\mathcal{F})$-tree is tame if and only if there exists $k\in\mathbb{N}^{\ast}$ such that $T$ is $k$-tame.
\end{prop}

\begin{proof}
Let $T$ be a tame minimal $(G,\mathcal{F})$-tree. By Theorem \ref{Levitt-good}, the tree $T$ splits as a graph of actions, all of whose vertex actions have dense orbits for the action of their stabilizer. As tame $(G,\mathcal{F})$-trees with dense orbits have trivial arc stabilizers by Lemma~\ref{dense-arcs}, we do not modify the collection of arc stabilizers of $T$ if we collapse all vertex trees to points. We can therefore reduce to the case where $T$ is simplicial. In this case, minimality implies that the $G$-action on $T$ has finitely many orbits of edges, from which it follows that $T$ is $k$-tame for some $k\in\mathbb{N}$. The converse statement will follow from the following proposition.
\end{proof}

\begin{prop}
For all $k\in\mathbb{N}^{\ast}$, there exists $\gamma(k)\in\mathbb{N}$ such that any $k$-tame minimal $(G,\mathcal{F})$-tree has at most $\gamma(k)$ orbits of directions at branch points.
\end{prop}

\begin{proof}
Let $T$ be a $k$-tame minimal $(G,\mathcal{F})$-tree. We first assume that $T$ is geometric. Let $\mathcal{G}$ be the dynamical decomposition of $T$. We recall that all vertex trees of $\mathcal{G}$ are either arcs or have dense orbits. As $T$ is geometric, there are finitely many orbits of directions in $T$ (see Corollary \ref{directions}). This implies that arc stabilizers are trivial in the vertex trees of $\mathcal{G}$ with dense orbits (Proposition \ref{dense-arcs}). If $x$ is a branch or inversion point of $T$ contained in one of the vertex trees with dense orbits $T_v$ of $\mathcal{G}$, then all directions at $x$ contained in $T_v$ have a positive contribution to the index $i(T)$. It follows from Proposition \ref{bound-index} that there is a bound on the number of such orbits of directions. Therefore, we can collapse all the vertex trees with dense orbits to points, and reduce to the case where $T$ is simplicial.

In this case, we argue by induction on $\text{rk}_K(G,\mathcal{F})$, and show that there is a bound $\gamma(k,l)$ on the number of orbits of directions in any $k$-tame minimal simplicial $(G,\mathcal{F})$-tree with $\text{rk}_K(G,\mathcal{F})\le l$. By splitting one of the vertex stabilizers of the splitting relatively to incident edge stabilizers if needed (which is made possible by Lemma \ref{unfold}, and can only increase the number of orbits of directions), we can assume that $T$ contains an edge $e$ with trivial stabilizer. By removing from $T$ the interior of the edges in the orbit of $e$ in $T$, we get one or two orbits of trees, whose stabilizers have a strictly smaller Kurosh rank. Let $T'$ be one of the trees obtained in this way, whose stabilizer we denote by $G'$. Then $T'$ is $k$-tame. If $T'$ is minimal, then we are done by induction. However, the tree $T'$ may fail to be minimal. The quotient graph $T'/G'$ consists of a minimal graph of groups $T'_{min}/G'$, with a segment $I=e_1\cup\dots\cup e_n$ attached to $T'_{min}/G'$ at one of its extremities (where we denote by $e_i$ the edges in $I$). All edge groups are cyclic, and they satisfy $G_{e_i}=G_{o(e_i)}\subseteq G_{e_{i+1}}$ for all $i\in\{1,\dots,n-1\}$. Since $T$ is $k$-tame, we have $n\le k$. By induction, there are at most $\gamma(k,l-1)$ orbits of directions in $T'_{min}$, so we get a uniform bound on the possible number of orbits of directions in $T$. 

We have thus shown that there is a uniform bound $\gamma(k)$ on the possible number of orbits of directions in a minimal $k$-tame geometric $(G,\mathcal{F})$-tree. Let now $T$ be a minimal nongeometric $(G,\mathcal{F})$-tree, and assume that $T$ has strictly more than $\gamma(k)$ orbits of directions. Arguing as in the proof of Proposition \ref{bound-index} in the nongeometric case, we can find a geometric $(G,\mathcal{F})$-tree $T'$ in which we can lift at least $\gamma(k)+1$ orbits of directions. Lemma \ref{lemma-abg} shows that the approximation $T'$ can be chosen to be $k$-tame. This is a contradiction.
\end{proof}

We finally establish one more condition under which a limit of tame trees is tame, which will be used in \cite{Hor14-6}.

\begin{prop}\label{limit-one-edge}
Let $(T_n)_{n\in\mathbb{N}}$ be a sequence of simplicial metric small $(G,\mathcal{F})$-trees that converges to a minimal $(G,\mathcal{F})$-tree $T$. If all trees $T_n$ contain a single orbit of edges, then $T$ is tame.
\end{prop}

We will make use of the following lemma.

\begin{lemma}\label{acylindrical}
Let $T$ be a simplicial small $(G,\mathcal{F})$-tree, with one orbit of edges. Then the fixed point set of any element of $G$ is a star of diameter at most $2$ for the simplicial metric on $T$.
\end{lemma}

\begin{proof}
Let $e_1$ and $e_2$ be two edges of $T$ stabilized by a common element $g\in G$. As $T$ has only one orbit of edges, there exists $h\in G$ such that $he_1=e_2$. By choosing an orientation on $e_1$, this relation defines an orientation on $e_2$. Then $hge_1=ghe_1$, which implies that $h$ commutes with $g$ because $T$ is small. Hence $h$ is elliptic in $T$. This implies that $e_1$ and $e_2$ point in opposite directions in $T$. As this is true of any pair of edges stabilized by $g$, the fixed point set of $g$ has the desired description.
\end{proof}

\begin{proof}[Proof of Proposition \ref{limit-one-edge}]
Up to possibly passing to a subsequence, one of the following situations occurs.
\\
\\
\textit{Case 1}: The length of the unique orbit of edges in $T_n$ converges to $0$.
\\
\\
In this case, we will show that $T$ is very small. This implies that $T$ is $1$-tame, and hence tame by Proposition \ref{good}. We have seen in the proof of Proposition \ref{very-small} that limits of small $(G,\mathcal{F})$-trees are small, and limits of trees with trivial tripod stabilizers have trivial tripod stabilizers. Let $g\in G$ be a nonperipheral element, and assume that there exists $l\ge 1$ such that $g^l$ fixes a nondegenerate arc $[a,b]\subseteq T$. Let $a_n$ (resp. $b_n$) be an approximation of $a$ (resp. $b$) in $T_n$. If $g$ were elliptic in $T_n$, then both $a_n$ and $b_n$ would be arbitrarily close to the fixed point set $X_n$ of $g^l$ in $T_n$, as $n$ goes to $+\infty$. It follows from Lemma \ref{acylindrical} that the diameter of $X_n$ in $T_n$ converges to $0$ as $n$ tends to $+\infty$. This contradicts the fact that $d_{T_n}(a_n,b_n)$ is bounded below (because $a\neq b$). Therefore, for $n$ large enough, the element $g$ is hyperbolic in $T_n$. The distances $d_{T_n}(a_n,g^l a_n)$ and $d_{T_n}(b_n,g^l b_n)$ converge to $0$, so the points $a_n$ and $b_n$ are arbitrarily close to the axis of $g$ in $T_n$, and $||g||_{T_n}$ converges to $0$. This implies that $d_{T_n}(a_n,ga_n)$ and $d_{T_n}(b_n,gb_n)$ both tend to $0$, so $g$ fixes $[a,b]$. 
\\
\\
\textit{Case 2}: There is a positive lower bound on the length of the unique orbit of edges in $T_n$.

Up to passing to a subsequence and rescaling $T_n$ by a factor $\lambda_n>0$ converging to some $\lambda>0$, we can assume that all trees $T_n$ have edge lengths equal to $1$. This implies that all translation lengths in $T_n$ belong to $\mathbb{Z}$, so all translation lengths in $T$ belong to $\mathbb{Z}$. It follows that $T$ is a simplicial metric tree (see \cite[Theorem 10]{Mor92}), so it has finitely many orbits of directions. Since a limit of small $(G,\mathcal{F})$-trees is small, the tree $T$ is tame. 
\end{proof}

\bibliographystyle{amsplain}
\bibliography{Horbez-biblio}

\end{document}